%% file: tdod_rev.tex
\documentclass[]{siamltex1213}
\usepackage{enumerate}
\usepackage{graphicx, caption, subcaption}
\usepackage{multirow, array}
\usepackage{xfrac}
\usepackage{amsmath, amsfonts, amssymb}
\usepackage{tikz}
\usepackage{pgfplots}
\usepackage{pgfplotstable}
\usetikzlibrary{calc}
\usepackage{standalone}
\newcommand*{\Scale}[2][4]{\scalebox{#1}{$#2$}}%
\newcommand{\reals}{\mathbb{R}}

\newcommand{\set}[1]{\{#1\}}

\newcommand{\norm}[1]{\lVert#1\rVert}
\newcommand{\half}{\Scale[0.5]{\tfrac{1}{2}}}
\newcommand{\nhalf}{\Scale[0.5]{-\tfrac{1}{2}}}
\renewcommand{\brack}[1]{\langle#1\rangle}
\newcommand{\inner}[2]{\left(#1, #2\right)}

\newcommand{\la}[1]{\ensuremath{\mathsf{#1}}}
\newcommand{\inv}[1]{\ensuremath{{#1}^{\Scale[0.5]{-1}}}}
\newcommand{\transp}[1]{\ensuremath{{#1}^{\Scale[0.5]{\top}}}}
\newcommand{\dual}[1]{#1^{\prime}}
\DeclareMathOperator{\spn}{span}
\makeatletter
\renewcommand{\inf}{\mathop{\@inf\vphantom{\@sup}}}
\renewcommand{\sup}{\mathop{\@sup\vphantom{\@inf}}}
\newcommand{\@inf}{\operatorname*{inf}}
\newcommand{\@sup}{\operatorname*{sup}}
\numberwithin{example}{section}
\newtheorem{remark}{Remark}
\numberwithin{remark}{section}

\newtheorem{conjecture}{Conjecture}
\numberwithin{conjecture}{section}

\usepackage[draft]{changes}  % put to final
\definechangesauthor[name={rev1}, color=red]{rev1}            % addresses rev1
\definechangesauthor[name={rev2}, color=green]{rev2}          % addresses rev2
\definechangesauthor[name={revboth}, color=magenta]{revboth}  % addresses both
\setremarkmarkup{(#2)}

\title{
Preconditioning trace coupled 3$d$-1$d$ systems using fractional Laplacian 
  \thanks{
Submitted to Numerical Methods for Partial Differential Equations
}
} 

\author{
Miroslav Kuchta \footnotemark[2] \and
Kent-Andre Mardal \footnotemark[2] \footnotemark[3] \and
Mikael Mortensen \footnotemark[2]
}

\begin{document}
% For affils . and addresses
\renewcommand{\thefootnote}{\fnsymbol{footnote}}
\footnotetext[2]{Department of Mathematics, Division of Mechanics, University of Oslo
\texttt{\{mirok, kent-and, mikaem\}@math.uio.no}
}
\footnotetext[3]{Center for Biomedical Computing, Simula Research Laboratory, 
}
\renewcommand{\thefootnote}{\arabic{footnote}}

\maketitle
\slugger{SISC}{xxxx}{xx}{x}{x--x}

\begin{abstract}
Multiscale or multiphysics problems often involve coupling of partial differential 
equations posed on domains of different dimensionality. 
In this work we consider a simplified model problem of a 3$d$-1$d$ coupling and the
main objective is to construct algorithms that may utilize standard multilevel 
algorithms for the 3$d$ domain, which has the dominating computational complexity.  
Preconditioning for a system of two elliptic problems posed, respectively, in a
three dimensional domain and an embedded one dimensional curve and coupled by 
the trace constraint is discussed. Investigating numerically the properties of the
well-defined \emph{discrete} trace operator, it is found that negative fractional Sobolev
norms are suitable preconditioners for the Schur complement of the system. The norms are employed 
to construct a robust block diagonal preconditioner for the coupled problem.    
\end{abstract}



\pagestyle{myheadings}
\thispagestyle{plain}
\markboth{Preconditioning for 3$d$-1$d$ coupled problems}{}

\input{introduction_rev.tex}
% Trace in 2d, theory for 3d continuous
\input{preliminaries_rev.tex}

% Sine/FEM - Schur complement indicating trace exponent
\input{trace_rev.tex}
% Coupled problems: Babuska, multiphysics
\input{coupled_rev.tex}
% Let cell1d be free!
\input{unmatch_rev.tex}
\input{conclusions_rev.tex}

\appendix

\input{appendix_rev.tex}
%
\bibliographystyle{siam}
\bibliography{tdod}
\end{document}

%% file: introduction_rev.tex
\section{Introduction}\label{sec:intro}

Let $\Omega$ be a bounded domain in 3$d$,  while $\Gamma$ represents a
1$d$ structure inside $\Omega$, and consider the following coupled problem    
\begin{subequations}
\label{eq:pde}
\begin{align}
\label{eq:pde1}
  -\Delta u + u + p\delta_{\Gamma} &= f &\mbox{ in } \Omega,\\
\label{eq:pde2}
-\Delta v + v - p &= g &\mbox{ on } \Gamma,\\
\label{eq:coupling}
T u - v  &=  h &\mbox{ on } \Gamma.
\end{align}
\end{subequations}
Here the term $p\delta_{\Gamma}$ is to be understood as a Dirac measure such that 
$\int_{\Omega}p(x)\delta_{\Gamma}w(x)\,\mathrm{d}x=\int_{\Gamma}p(t)w(t) \,\mathrm{d}t$
for a continuous function $w$. 
We remark that from a mathematical point of
view the trace $T$ of $u$ required in \eqref{eq:coupling} is in the continuous case 
not well-defined unless the functions are sufficiently regular. For simplicity of
implementation the system shall be considered with homogeneous Neumann boundary conditions.

The system \eqref{eq:pde} is relevant in numerous biological applications where the
embedded (three dimensional) structure is such that order reduction techniques 
can be used to capture its response by a one dimensional model. Equation
\eqref{eq:pde1} then models processes in the bulk, while \eqref{eq:coupling} is
the coupling between the domains. A typical example of such a system is a vascular 
network surrounded by a tissue and the order reduction is due to the employed assumption of 
radii of the arteries being negligible in comparison to their lengths. To list a
few concrete applications, the 3$d$-1$d$ models have been used, e.g., in 
\cite{grinberg, linninger, fang, reichold} to study blood and oxygen transport in 
the brain or in \cite{cattaneo} to describe fluid exchange between microcirculation and 
tissue interstitium. Efficiency of cancer therapies delivered through
microcirculation was studied in \cite{cattaneo_tumor}, and hyperthermia as a cancer treatment in \cite{nabil}. We note that the employed
models are more involved than \eqref{eq:pde}, but that 
the system still qualifies as a relevant model problem. 

% D'Angelo stuff 
Due to the Dirac measure term and the three-to-one dimensional trace operator, the
problem \eqref{eq:pde} is not standard and establishing its well-posedness is a
delicate issue. In fact, considering \eqref{eq:pde1} with a known $p$ and homogeneous 
Dirichlet boundary conditions, the equation is not solvable in $H_0^1(\Omega)$, as 
$\nabla u$ may be unbounded in the neighborhood of $\Gamma$. 
A similar problem was studied in \cite{quarteroni}, where two elliptic problems
were coupled via a Dirac measure source term, and a unique weak solution was
found using weighted Sobolev spaces. In particular, the weighted spaces that include
a distance function ensured that the trace could be
defined as a bounded operator. A corresponding finite element method (FEM) for the problem was
discussed in \cite{dangelo}, where optimal convergence in the weighted Sobolev norm
was shown using graded meshes.
Optimal convergence of FEM with regular meshes is proved in \cite{koppl} and \cite{koppl2016local}
for the elliptic problems with point singular data and line singular data and
the $L^2$ norm outside of the fixed neighborhood of the singularity. Therein, the existence
of the weak solution relies on spaces $W^{1, p}(\Omega)$, $1\leq p < 2$.

We remark that the weighted Sobolev spaces in \cite{dangelo, quarteroni}
and the Sobolev spaces $W^{1, p}(\Omega)$ in \cite{koppl, koppl2016local} were introduced
in the analysis of the continuous problems, however, standard finite elements
were used in the implementation.

While standard finite element methods provide accurate discretization in the alternative norms
of the above mentioned nonstandard Sobolev spaces, this does not imply that standard preconditioning
algorithms will be efficient. In fact, as described e.g. in \cite{kent_ragnar}, the construction of preconditioners
is deeply connected with the mapping properties of the underlying continuous differential operators
and to the authors knowledge the efficiency in the Sobolev spaces with distance functions have not been analyzed.   
Hence, the use of the weighted Sobolev spaces
has prevented the construction of efficient solution algorithms
and the more application oriented works 
\cite{cattaneo_tumor, cattaneo, nabil}, that build on the analysis in \cite{dangelo, quarteroni}, 
relied on incomplete LU preconditioning. 
To resolve this problem, we have in this paper taken an 
alternative approach where standard multilevel algorithms for elliptic problems are reused for the
3$d$ problem. This approach does however require that novel algorithms are developed for the 1$d$ problem. 
The special construction of algorithms for the 1$d$ problem is justified by the fact
that, in general, the computational complexity of a 1$d$ problem is low compared to 
a 3$d$ problem. We shall illustrate this fact by several numerical experiments.  
% \KAM{some minor changes in the 3 paragraphs above}

The current paper is an extension of \cite{2d-1d}, where a system similar to 
\eqref{eq:pde1}--\eqref{eq:coupling} was analyzed for the case $\Omega$ a bounded domain in 
2$d$ and $\Gamma$ a structure of codimension one. Therein, robust preconditioners were established, based 
on the operator preconditioning framework \cite{kent_ragnar}, in which preconditioners are
constructed as approximate Riesz mappings in properly chosen Hilbert spaces. The framework
often allows for construction of order-optimal preconditioners, with convergence independent of  
material and discretization parameters, directly from the analysis of the continuous system of equations.
In particular, in \cite{2d-1d} it was shown that the proper preconditioning
relied on a nonstandard fractional $H^{\nhalf}$ inner product. Crucial for the analysis was
the fact that the trace operator $T$ is a well-defined mapping between $H^1(\Omega)$ and 
$H^{\half}(\Gamma)$, when $\Gamma$ is of codimension one with respect to $\Omega$. 

The case when the trace operator $T$ maps functions defined on $\Omega$ to $\Gamma$ and $\Gamma$ is of codimension
two is challenging as the properties of the trace operator are not established from a theoretical point of view.
If we assume some additional regularity such that $u\in H^{1+\epsilon}$ then $Tu\in H^\epsilon$ for $\epsilon>0$,
see e.g. \cite{ding1996proof}. However, the result is known to break down in the limit when $\epsilon=0$. 
We therefore propose to weaken the requirements on $T$ and instead consider
$T$ as a mapping between  $H^{1}(\Omega)$ and $H^{s}(\Gamma)$ for some $s<0$ that will be determined.   
To investigate the existence of such a $s$ we perform a comprehensive numerical study with 
various discretizations; that is, finite element methods with conforming/non-conforming elements and matching/non-matching meshes, 
and by considering the Galerkin method with eigenbasis of Laplace operator. We demonstrate that all of these
different methods point to the construction of the same preconditioning operator, namely
$(-\Delta)^{s}$,  where $s\in\left(-0.2, -0.1\right)$ and the range seems to be independent
  of the discretization method.
We demonstrate numerically that this choice defines a good preconditioner
for problems with complex 1$d$ geometries and for 3$d$ meshes that 
are both highly refined or rather coarse close to the 1$d$ mesh as long
as the mesh is shape-regular and the discrete problem is invertible.
% }
% {and an optimal $s$ is in a narrow range $s\in [-0.1,-0.2]$.
% We demonstrate numerically that the  minimum is maintained for rather 
% complex 1$d$ geometries and for 3$d$ meshes that 
% are both highly refined or rather coarse close to the 1$d$ mesh as long
% as the mesh is shape-regular and the trace is surjective.}
% \KAM{Shouldn´t we put bounds like $s\in (-0.1,-02)$? Surely $s$ can not be arbitrary large. 
% I also think that the word surjective is a red flag. It could easily be
% interpreted in a functional analysis way.}

Our work is structured as follows. In \S\ref{sec:prelims} the theoretical
background is presented. Section \ref{sec:trace} discusses numerical experiments
using spectral and finite element discretizations that identify suitable norms 
for the discrete 3$d$-1$d$ trace operator. In \S\ref{sec:coupled} the identified norms are
employed to construct optimal preconditioners for coupled model 3$d$-1$d$ problems
discretized with FEM and matched discretizations of $\Omega$ and $\Gamma$. In \S\ref{sec:unmatch} 
this restriction is lifted, the corresponding inf-sup condition is discussed, and
we present numerical experiments that suggest the identified norms lead to
good preconditioners. Finally, conclusions are summarized in \S\ref{sec:conclusions}.
%\KAM{Mention 2D-1D more clearly}

%% file: preliminaries_rev.tex
\section{Notation and preliminaries}\label{sec:prelims}
% Notation
Let $X$ be a Hilbert space of functions defined on a domain $D\subset\reals^d$,
$d=1, 2, 3$. The norm of the space is denoted by $\norm{\cdot}_X$, while
$\brack{\cdot, \cdot}_{\dual{X}, X}$ is the duality pairing between $X$ and its
dual space $\dual{X}$. We let $(\cdot, \cdot)_X$ denote the inner product of
$X$, while, to simplify the notation, $(\cdot, \cdot)_D$ is the $L^2$ inner product.
The Sobolev space of functions with $m$ square integrable derivatives is
$H^m(D)$. Finally, $H_0^m(D)$ denotes the closure of the space of smooth functions 
with compact support in $D$ in the $H^m(D)$ norm. We will also employ 
Sobolev spaces with fractional derivatives, which are more precisely defined later.  

We use normal capital font to denote operators over infinite dimensional spaces,
e.g. $A:X\rightarrow \dual{X}$.
If $A:X \rightarrow Y$ is a bounded operator we let $\dual{A}:\dual{Y}\rightarrow\dual{X}$
denote the adjoint operator $\brack{\dual{y}, Ax}_{\dual{Y}, Y}=\brack{\dual{A}y, x}_{\dual{X}, X}$, $y\in\dual{Y}$, $x\in X$.
For a discrete subspace $X_h\subset X$, $\dim
X_h=n$, the subscript $h$ is used to distinguish the finite dimensional operator due 
to the Galerkin method, e.g., $A_h:X_h\rightarrow\dual{X_h}$ defined by
\[
  \brack{A_h u_h, v_h}_{\dual{X}, X} = 
  \brack{A u, v_h}_{\dual{X}, X}\quad u_h, v_h\in X_h\mbox{ and }u\in X.
\]
For a given basis, $\set{\phi_i}_{i=1}^n$ of $X_h$, the matrix representation of
the operator is denoted by sans serif font. Thus $A_h$ is represented by 
$\la{A}\in\reals^{n\times n}$ with entries
\[
  \la{A}_{i, j} = \brack{A_h \phi_j, \phi_i}_{\dual{X}, X}.
\]
The function $u_h\in X_h$ is represented in the basis by a coefficient vector
$\la{u}\in\reals^n$, where $u_h=\la{u}_i\phi_i$ (summation convention invoked).
Finally, for the inner product of vectors $\la{u}, \la{v}$ in $\reals^n$
shall be denoted as $\transp{\la{u}}\la{v}$.

% Review the trace properties
\subsection{Properties of the trace operator}\label{sec:trace_theory}
We consider $\Omega\subset\reals^d$ an open connected domain
with Lipschitz boundary $\partial\Omega$ and $\Gamma$ a Lipschitz submanifold of 
codimension one or two in $\Omega$. The trace operator $T$ is defined by 
$T u = u|_\Gamma$ for $u\in C(\overline\Omega)$. 

In case the codimension of $\Gamma$ is one, the properties of the trace operator
are well known. In particular, $T:H^s(\Omega)\rightarrow H^{s-\half}(\Gamma)$ is
bounded and surjective, see, e.g., \cite[ch. 7]{AdamsFournier} for $s>\tfrac{1}{2}$.
As a direct consequence we then have that the trace to $\Gamma$ of codimension
two is well behaved as mapping from $H^{1+\epsilon}(\Omega)$ to $H^{\epsilon}(\Gamma)$
for any $\epsilon > 0$, cf.
\cite{renardy2006introduction} and \cite{ding1996proof} for the case of
unbounded and bounded domains respectively.
However, for $\epsilon=0$, it is known that the trace operator is unbounded 
as a mapping between $H^1(\Omega)$ and $L^2(\Gamma)$. We therefore conjecture
that the trace is well-behaved between $H^1(\Omega)$ and $H^{s}(\Gamma)$, $s<0$. 

The fractional Sobolev space $H^{s}(\Gamma)$  shall be defined by interpolation
\cite{LionsMagenes, arioli_siam}. For the sake of completeness, we review here
the presentation from \cite{2d-1d}. Let $u, v\in X=H^1(\Gamma)$.
For $u$ fixed $v \mapsto (u, v)_{\Gamma}$ is in $\dual{X}$ and by the Riesz-Fr\'echet 
theorem there is a unique $w\in X$ such that $(w, v)_X = (u, v)_{\Gamma}$ for
any $v\in X$. The operator $S:u\rightarrow w$ is injective and compact and thus
the eigenvalue problem $S\phi_i=\lambda_i \phi_i$ (no summation implied) is well-defined. 
In addition, $S$ is self-adjoint and positive-definite such that the eigenvalues form 
a nonincreasing sequence $0 < \lambda_{k+1} \leq \lambda_{k}$ and $\lambda_{k}\rightarrow 0$. 
By definition, the eigenvectors satisfy
\[
  (\phi_i, v)_X=\lambda_i^{-1}(\phi_i, v)_{\Gamma}\quad v\in X,
\]
or equivalently
\begin{equation}\label{eq:snorm_construction}
  A\phi_i = \lambda_i^{-1}M\phi_i\mbox{ with }
  \brack{Au, v}_{\dual{X}, X}=(u, v)_X\mbox{ and }
  \brack{Mu, v}_{\dual{X}, X}=(u, v)_{\Gamma}.
\end{equation}
Further, the set of eigenvectors $\set{\phi_k}_{k=1}^{\infty}$ forms a basis of 
$X$, which is orthogonal in the inner product of $X$ and orthonormal in the
$L^2(\Gamma)$ inner product. Finally, for $s\in [-1,1]$ we define the $s$-norm
of $u = c_k \phi_k \in \spn\set{\phi_k}_{k=1}^\infty$ as
\begin{equation}\label{eq:snorm_contin}
  \norm{u}_{H^s(\Gamma)} 
  = 
  \sqrt{c^2_k \lambda^{-s}_k}.
\end{equation}
The space $H^s(\Gamma)$ is finally defined as the closure of the $\spn\set{\phi_k}_{k=1}^\infty$
in the $s$-norm, while $H^s_{0}(\Gamma)$ is then defined analogously to
$H^s(\Gamma)$ with $X=H^1_0(\Gamma)$ in the construction. 

Following the approach in \cite{2d-1d}, a weak formulation of the homogeneous 
Dirichlet problem for \eqref{eq:pde1}--\eqref{eq:coupling} with $\Omega\in\reals^3$, 
$\Gamma\subset\Omega$ of codimension \emph{one} or \emph{two}, using the method of Lagrange multipliers, 
reads: Find  $\left(u, v, p\right)\in H_0^1(\Omega) \times H_0^1(\Gamma) \times Q$ 
such that 
\begin{equation}
  \label{eq:weak_form}
\begin{aligned}
  \inner{\nabla u}{\nabla \phi}_{\Omega} + \inner{u}{\phi}_{\Omega}+\inner{p}{T \phi}_\Gamma &=
  \inner{f}{\phi}_{\Omega} 
  \quad &\forall\phi &\in H^1_0(\Omega),
  \\
  \inner{\nabla v}{\nabla \psi}_{\Gamma} + \inner{v}{\psi}_{\Gamma} - \inner{p}{\psi}_\Gamma &=
  \inner{g}{\psi}_{\Gamma}
  \quad &\forall\psi &\in H^1_0(\Gamma),
  \\
  \inner{\chi}{T u - v}_\Gamma &= \inner{h}{\chi}_\Gamma
  \quad &\forall\chi &\in Q.
\end{aligned}
\end{equation}
For $\Gamma$ of codimension one, the problem is well-posed with $Q=H^{\nhalf}_{0}$
owing to the fact that $T:H^1_0(\Omega) \rightarrow H^{\half}_{0}(\Gamma)$ 
is an isomorphism. Similarly, for $\Gamma$ of codimension two, the well-posedness hinges
on whether $T:H^1_0(\Omega) \rightarrow \dual{Q}$ is an isomorphism for some space $Q$.
Hovewer, to the best of the authors knowledge this result is not known. In this paper,
we therefore conjecture that the space $Q$ is closely related to $H^{s}_{0}(\Gamma)$ for
some suitable $s<0$.

Assuming the the conjecture holds, the operator $\mathcal{A}$ defined by
\eqref{eq:weak_form}
\begin{equation}
\mathcal{A} \left[ \begin{array}{c} u\\ v\\ p\\ \end{array} \right] = 
\left[ \begin{array}{ccc} 
       I_{\Omega}-\Delta_{\Omega} & 0 & T' \\
       0 & I_{\Gamma}-\Delta_{\Gamma} & -I_{\Gamma}  \\
       T & -I_{\Gamma} & 0  
       \end{array} \right]
        \left[ \begin{array}{c} u\\ v\\ p\\ \end{array} \right] = 
        \left[ \begin{array}{c} f\\ g\\ h\\ \end{array} \right]  .  
\end{equation}
is an isomorphism mapping $H^1_0(\Omega) \times H^1_0(\Gamma) \times Q$
to its dual space and a proper preconditioner can be formed as
\begin{equation}
\mathcal{B} = 
\left[ \begin{array}{ccc} 
       (I_{\Omega}-\Delta_{\Omega})^{-1} & 0 & 0 \\
       0 & (I_{\Gamma}-\Delta_{\Gamma})^{-1} & 0  \\
       0 & 0 & R_Q  
       \end{array} \right], 
\end{equation}
where $R_Q$ is the Riesz mapping between the dual of $Q$ and $Q$, cf. \cite{kent_ragnar}.
As the Riesz mapping is not easily obtained from the analysis of the continuous
problem, we shall in the following resort to investigating the mapping properties
of the trace operator of codimension two by a series of numerical experiments
with different spaces for $Q$.  

Let now $V_h\subset H^1(\Omega)$. Considering \eqref{eq:pde} on the finite dimensional 
spaces, we obtain a variational problem: Find $\left(u_h, v_h, p_h\right)\in V_h \times W_h \times Q_h$ 
such that 
\begin{equation}
  \label{eq:weak_form_h}
\begin{aligned}
  \inner{\nabla u_h}{\nabla \phi_h}_{\Omega} + \inner{u_h}{\phi_h}_{\Omega} + 
  \inner{p_h}{T_h \phi_h}_\Gamma &=
  \inner{f}{\phi_h}_{\Omega} 
  \quad\quad&\phi_h &\in V_h,
  \\
  \inner{\nabla v_h}{\nabla \psi_h}_{\Gamma} + \inner{v_h}{\psi_h}_{\Gamma} -
  \inner{p_h}{\psi_h}_\Gamma &=
  \inner{g}{\psi_h}_{\Gamma}
  \quad\quad&\psi_h &\in W_h,
  \\
  \brack{\chi_h,  T_h u_h - v_h}_\Gamma &= \inner{h}{\chi_h}_\Gamma
  \quad\quad&\chi_h &\in Q_h.
\end{aligned}
\end{equation}
Here the discrete trace operator $T_h$ is well-defined as the functions in $V_h$
are continuous. In the absence of existence
result for the continuous problem the discrete preconditioner
cannot be constructed within the framework of operator precoditioning,
i.e. as a discretization of a suitable Riesz mapping. We therefore adapt a
different framework, namely the matrix Schur complement preconditining
\cite{benzi2005numerical, murphy}. That is, we attempt to construct the
preconditioner for \eqref{eq:weak_form_h} by reasoning directly about the
properties of the discrete systems.

From a linear algebra point of view, the problem \eqref{eq:weak_form_h} is a saddle-point 
system
\[
  \begin{bmatrix}
    \la{A} & \transp{\la{B}} \\ 
    \la{B} & 0
  \end{bmatrix}
  \begin{bmatrix}
\la{x} \\ \la{y}  
  \end{bmatrix}
=
\begin{bmatrix}
\la{b} \\ \la{c}  
\end{bmatrix},
\]
with $\la{A}$ a symmetric positive definite matrix. In case $\la{B}$ has a
full row rank, the discrete problem is uniquely solvable and block diagonal
preconditioner can be constructed as an approximate inverse of the matrix $\text{diag}(\la{K}, \la{L})$, 
% \[
%   \begin{bmatrix}
% \la{K} &  \\ 
%        & \la{L} 
%   \end{bmatrix}
% \]
where $\la{K}$ should be spectrally equivalent with $\la{A}$ and $\la{L}$ should be 
spectrally equivalent with the Schur complement $\la{B}\inv{\la{A}}\transp{\la{B}}$, 
see, e.g., \cite{rusten, silvester1994fast}. Considering
\eqref{eq:weak_form_h}, the key question is thus whether it is possible 
(in an efficient and systematic manner) to construct an operator that is spectrally 
equivalent with the Schur complement. Motivated by the 2$d$-1$d$ problem and our
conjectured mapping properties of the trace the operator shall be based
on the norm of the $H^s(\Gamma)$ space \eqref{eq:snorm_contin}.

Following \cite{2d-1d}, the discrete approximation of the $s$-norm  
shall be constructed by mirroring the continuous eigenvalue problem
\eqref{eq:snorm_construction}. More specifically, let $X_h\subset X$ and matrices 
$\la{A}$, $\la{M}$ be the representations of $A_h$, $M_h$; the Galerkin
approximations of operators $A$, $M$ from \eqref{eq:snorm_construction}. Then
there exists an invertible matrix $\la{U}$ and diagonal, positive-definite matrix
$\la{\Lambda}$ satisfying $\la{A U}=\la{M U \Lambda}$. Moreover, the product
$\transp{\la{U}}\la{M}\la{U}$ is an identity such that the columns of
$\la{U}$ form an $\la{A}$ orthogonal and $\la{M}$ orthonormal basis of $\reals^n$. 
In order to define the discrete norm, we let $\la{H}_s$ be a symmetric,
positive-definite matrix
\begin{equation}\label{eq:smat}
  \la{H}_s = \transp{\left(\la{MU}\right)}\la{\Lambda}^s\left(\la{MU}\right).
\end{equation}
The matrices $\la{H}_{s, 0}$ are defined analogously to \eqref{eq:smat}, using the
eigenvalue problem for the Laplace operator with homogeneous Dirichlet boundary
conditions. For $u_h\in X_h$ represented in the basis of the space by a 
coefficient vector $\la{u}$, let $\la{c}$ be the representation of $\la{u}$ in the 
basis of eigenvectors, that is, $\la{u}=\la{U}\la{c}$. We then set
\begin{equation}\label{eq:snorm}
  \norm{u_h}_{H^s(\Gamma)} = 
  \sqrt{\transp{\la{u}}\la{H}_s\la{u}} =
  \sqrt{\transp{\la{c}}\la{\Lambda}^s\la{c}}.
\end{equation}
% \KAM{Explain c also with words.}

% The generalized eigenvalue problem required for evaluating the discrete $s$-norm 
% \eqref{eq:snorm} becomes trivial if the approximation space $V_n$ is such that 
% $V_n=\spn\set{\phi_i}_{i=1}^{n}$, i.e. the basis is formed by the eigenvectors of 
% the continuous problem \eqref{eq:snorm_construction}. Such a discretization is 
% practically limited to Cartesian domains, however, it will prove useful in studying 
% the trace operator when the codimension of $\Gamma$ in $\Omega$ is two. 

% In case $\Omega\subset\reals^3$ and $\Gamma$ is of codimension 2 extending the
% trace operator $T u = u|_\Gamma$ for $u\in C(\overline\Omega)$ to Sobolev spaces
% requires improved regularity. For $m>\tfrac{5}{2}$ the space $H^m(\Omega)\subset C(\Omega)$ 
% by Sobolev inequality/embedding, e.g. \cite[ch. 5]{evans}, and so taking point
% values of the function on the interior of the curve is well-defined. \MK{What
% would happen in $H^2$, what about $H^{1+\epsilon}$? Can the trace space be
% characterized as some fraction Sobolev space?} 
% \KAM{We do not go into this. We do not know it well enough currently. 
% There is a huge bulk of results on this that we do not know well}

%% file: trace_rev.tex
\section{Norms for the discrete 3$d$-1$d$ trace}\label{sec:trace}

The matrices $\la{H}_s$ shall be employed to construct a preconditioner for the
Schur complement of the system \eqref{eq:weak_form_h}. Considering \eqref{eq:weak_form},
the matrix is a sum of two parts which correspond respectively to operators
$T(-\Delta_{\Omega}+I_{\Omega})^{-1}\dual{T}$ and $I_{\Gamma}(-\Delta_{\Gamma}+I_{\Gamma})^{-1}\dual{I_{\Gamma}}$.
As the matrix stemming from the latter term is by definition spectrally equivalent
with $\la{H}_{-1}$ we shall next focus only on the former trace term. We note that
if our conjucture on the mapping properties of the trace operator holds, that
is $T:H^1(\Omega)\rightarrow H^s(\Gamma)$ is bounded and sujective for some $s<0$,
then  $T(-\Delta_{\Omega}+I_{\Omega})^{-1}\dual{T}: \dual{H^s(\Gamma)} \rightarrow H^s(\Gamma)$
is an isomorphism and the preconditioner could be realized by the fractional
norm matrix.

To investiagate the conjectured spectral equivalence of the trace term,
let $V, Q$ be the spaces of continuous functions over $\Omega$ and $\Gamma$
respectively and consider the problem of minimizing $v\mapsto (\nabla v, \nabla v)_{\Omega}-2(f, v)_{\Omega}$, 
$v\in V$, subject to $v=0$ on the boundary and the constraint $T v=g$ on $\Gamma$.
The minimization problem leads to the variational problem for $u\in V$, $p\in Q$
satisfying
\begin{equation}\label{eq:trace_problem}
  \begin{aligned}
    (\nabla u, \nabla v)_{\Omega} + (p, T v)_{\Gamma} &= (f, v)_{\Omega} &\quad\forall v\in V,\\
    (q, T u)_{\Gamma} &= (q, g)_{\Gamma} &\quad\forall q\in Q 
  \end{aligned}.
\end{equation}
The Schur complement of \eqref{eq:trace_problem} is thus closely related
to the critical trace term in the Schur complement of \eqref{eq:weak_form}.

Using finite dimensional subspaces of $V_h\subset V$ and $Q_h\subset Q$ \footnote{
  We use the same subscript to signify that the function spaces cannot be
  arbitrary and instead must satisfy inf-sup compatibility condition.
}
the problem \eqref{eq:trace_problem} is equivalent to the linear system
\begin{equation}\label{eq:sine_3d_system}
  \begin{bmatrix}
    \la{A} & \transp{\la{T}}\\
    \la{T} & 0
  \end{bmatrix}
  \begin{bmatrix}\la{u}\\
    \la{p}\end{bmatrix}
  =
  \begin{bmatrix}\la{f}\\
    \la{g}\end{bmatrix}.
\end{equation}
and we wish to find computational evidence for the following claim.

\begin{conjecture}\label{thm:equiv}
There exist $s<0$ and constants $0<\lambda_{*}\leq\lambda^{*}$ such that
for any $h>0$
\begin{equation}\label{eq:equiv_cond}
  \lambda_{*} \leq \frac{\transp{\la{x}}\left(\la{T}\inv{\la{A}}\transp{\la{T}}\right)\la{x}}
  {\transp{\la{x}}\la{H}_{s, 0}\la{x}} \leq \lambda^{*}.
\end{equation}
\end{conjecture}

In addition to spectral equivalence condition \eqref{eq:equiv_cond} we shall
also consider a weaker requirement, where we wish to find $s$ for which the
condition number of the preconditioned Schur complement is bounded in $h, H$ for some $s<0$.
More precisely, let $0<\lambda_{\text{min}}(s, h)\leq\lambda_{\text{max}}(s, h)$ be the smallest and
largest eigenvalues of the generalized eigenvalue problem
\begin{equation}\label{eq:cond_eigs}
  \left(\la{T}\inv{\la{A}}\transp{\la{T}}\right) \la{p} = \lambda \la{H}_s \la{p}.
\end{equation}
%We shall investigate whether \ref{thm:cond} holds.

\begin{conjecture}\label{thm:cond}
There exist $s<0$ such that the condition 
number
\begin{equation}\label{eq:cond_cond}
  \kappa(s, h) = \frac{\lambda_{\text{max}}(s, h)}{\lambda_{\text{min}}(s, h)} \leq C \quad \forall h > 0,
\end{equation}
for some constant $C$.
\end{conjecture}

We note that the condition \eqref{eq:cond_cond} is motivated by the
fact that convergence of the preconditioned conjugate gradient method is estimated 
in terms of the condition number, see, e.g., \cite{trefethen}. For suitable $s$ the 
linear system with the Schur complement could thus be solved efficiently.
We also note that the condition is weaker than spectral equivalence \eqref{eq:equiv_cond}. 

To investigate conjectures \ref{thm:equiv}, \ref{thm:cond} we let 
$\Omega=\left[0, 1\right]^3$ and choose $\Gamma$ as simple straight lines;
$\Gamma_1=\set{(t, \tfrac{1}{2}, \tfrac{1}{2}); t\in\left[0, 1\right]}$ 
and $\Gamma_2=\set{(t, t, t); t\in\left[0, 1\right]}$. For discretization
of \eqref{eq:trace_problem} the discrete subspaces shall be first constructed
using the basis of eigenfunctions of the Laplacian.

%In Example \ref{ex:sine_2d} a priori knowledge of the trace space lead to an
%optimal preconditioner for the model problem \eqref{eq:sine_2d}. In particular,
%the norm of the trace space was used to construct a spectrally equivalent operator
%to the Schur complement of \eqref{eq:sine_2d_system}. 

% \KAM{need an intro paragraph of this section}

% For $\Omega\subset\reals^3$ 
% and $\Gamma$ a one dimensional curve, the trace space is not a priori known and
% we shall therefore attempt to characterize it numerically. To this end, we shall
% at first use the spectral discretization and search for the $s$-norm \eqref{eq:snorm} 
% for which the condition number of the preconditioned Schur complement is bounded 
% in the discretization parameter. We note that the condition is motivated by the
% fact that convergence of the preconditioned conjugate gradient method is estimated 
% in terms of the condition number, see, e.g., \cite{trefethen}. For suitable $s$ the 
% linear system with the Schur complement could thus be solved efficiently.
% We also note that the condition is weaker than spectral equivalence. In fact, if 
% such $s$ exists, the matrix $\la{H}_{s, 0}$ is spectrally equivalent with the Schur 
% complement if and only if one of the extremal eigenvalues is bounded by a constant. 

\subsection{Trace operator with spectral discretization}\label{sec:trace_spectral}
Let $\set{\phi_k}_{k\geq 1}$ be the set of eigenvectors of the Laplace operator on unit interval with
homogeneous Dirichlet boundary conditions and set $Q_m=\spn\set{\phi_k}_{k=1}^{m}$ while
the $n^3$ dimensional space $V_n$ of functions on $\Omega$ shall be defined as a
tensor product.

Considering \eqref{eq:trace_problem} with spaces $V_n$, $Q_m$ the matrix $\la{A}$ in
\eqref{eq:sine_3d_system} diagonal. The trace matrix $\la{T}\in\reals^{{m\times n^3}}$
for curve $\Gamma_1$ is sparse with entries
  \[
\la{T}_{j, (i, k, l)} = 
  \begin{cases}
  0 & k\mbox{ or }l\mbox{ even }\\
  (-1)^{k+1}(-1)^{l+1}2\delta_{i j} & \mbox{otherwise}
  \end{cases}.
  \]
Note that for $m>n$ the matrix does not have a full row rank and the system is singular.
We therefore set $m=n$. For $\Gamma_2$ the trace matrix is sparse with a more
involved sparsity pattern and at most four nonzero entries per row
  \[
\la{T}_{j, (i, k, l)}
  =
4\sqrt{3}\int_0^1\sin{j\pi t}\sin{i\pi t}\sin{k\pi t}\sin{l\pi t}\,\mathrm{d}t.
  \]
%
  
%Finally, we consider the generalized eigenvalue problem for the Schur complement 
%of \eqref{eq:sine_3d_system} and matrices $\la{H}_{s, 0}$ where such exponents are of 
%interest, for which the spectral condition number $\kappa=\lambda_{\max}/\lambda_{\min}$ is 
%bounded in the discretization parameter. Note that with $\Gamma_1$ the Schur
  Having defined the terms in \eqref{eq:sine_3d_system} we consider the
  generalized eigenvalue problem \eqref{eq:cond_eigs} for different values of
  $s$ and the discretization parameter $n$. Observe that in case of $\Gamma_1$
  the Schur complement can be obtained in a closed form. Indeed, the matrix is
  diagonal $S_j\delta_{ij}$ (no summation implied) with entries
\begin{equation}\label{eq:schur_3d}
S_j = \frac{4}{\pi^2}\sum_{l, m\mbox{{ odd}}}^n\frac{1}{j^2+l^2+m^2}.
\end{equation}
For $\Gamma_2$ the matrix is dense and shall be computed from the definition
$\la{T}\inv{\la{A}}\transp{\la{T}}$. As such a smaller $n$ is explored in this configuration.

\begin{figure}
  \begin{center}
    \includegraphics[width=0.49\textwidth,
    height=0.25\textheight]{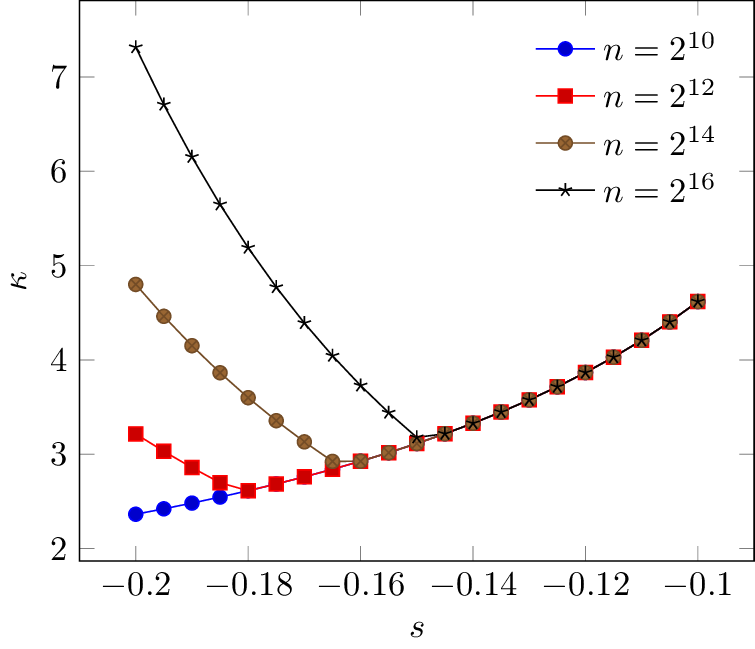}
    \includegraphics[width=0.49\textwidth,
    height=0.25\textheight]{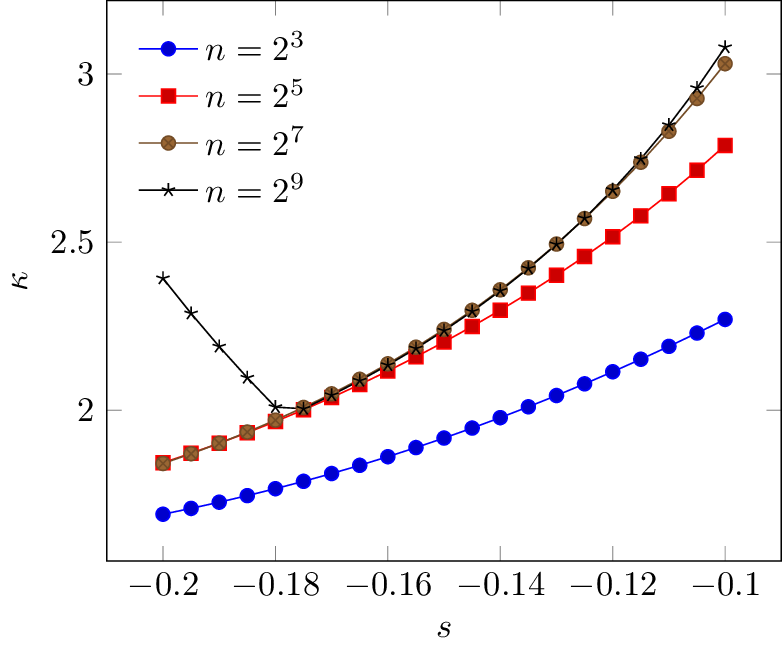}
    \caption{
    Spectral condition numbers \eqref{eq:cond_cond} computed from the generalized eigenvalue problem
    for Schur complement of \eqref{eq:sine_3d_system} and matrices $\la{H}_{s,
    0}$, see \eqref{eq:snorm}. (Left) The constraint is considered on 
    $\Gamma_1=\set{(t, \tfrac{1}{2}, \tfrac{1}{2}); t\in\left[0, 1\right]}$.
    (Right) $\Gamma_2=\set{(t, t, t); t\in\left[0, 1\right]}$ is considered.
    % With both configurations, values $s$ close to $-0.14$ yield bounded $\kappa$.
    }
    \label{fig:sine_3d}
  \end{center}
\end{figure}
The results of the numerical experiments with $s\in\left[-0.2, -0.1\right]$ are 
summarized in Figure \ref{fig:sine_3d}. We observe that values
$s\in\left[-0.145, -0.1\right]$ yield bounded condition numbers for $\Gamma_1$.
The condition numbers are not quite converged for the other configuration,
however, it is possible to identify unstable exponents $s<-0.18$. Moreover, the values
close to $s=-0.14$ appear to be stable also in this configuration. This fact is 
easier to appreciate in Table \ref{tab:sine_3d}, which shows $\lambda_{\min}$, 
$\lambda_{\max}$ and $\kappa$ as functions of the discretization parameter for 
$s=-0.14$. For $\Gamma_1$ the condition number is evidently constant, while for 
$\Gamma_2$ the number appears to be bounded. The observation are therefore
supportive of conjecture \ref{thm:cond}. %claim \eqref{eq:cond_cond}.

For neither of the configurations and any of the considered vales $s$ the
smallest and largest eigenvalues are bounded and thus, contrary to conjecture \ref{thm:equiv},
the matrices $\la{H}_{s, 0}$ are not spectrally equivalent with the Schur complement.
However, taking e.g. $s=-0.14$, either of $\lambda_{\min}(n)$, $\lambda_{\max}(n)$ defines a
mesh-depenedent scale $\tau(n)\la{H}_{-0.14, 0}$ that yields spectral
equivalence. Such scale, however, is not easily computable in general as it
involves the inverse of the 3$d$ problem.

% data for -0.14
\begin{table}
  \begin{center}
  \caption{
Smallest and largest eigenvalues $\lambda_{\min}$, $\lambda_{\max}$ and the
spectral condition numbers $\kappa$ of \eqref{eq:cond_eigs}. (Top) The
preconditioner is $\la{H}_{-0.14, 0}$.
While the eigenvalues are unbounded the condition number is bounded in $n$.
(Bottom) Matrix $\la{H}_{0, 0}$ (identity matrix) is used as the preconditioner.
In agreement with the analysis in Remark \ref{rem:s_0}, constant $\lambda_{\min}$ 
and $\lambda_{\max}$ with a logarithmic growth are observed.
}
\footnotesize{
\begin{tabular}{cccc||cccc}
  \hline
\multicolumn{4}{c||}{$\Gamma_1=\set{(t, \tfrac{1}{2}, \tfrac{1}{2}); t\in\left[0, 1\right]}$} 
  & 
\multicolumn{4}{c}{$\Gamma_2=\set{(t, t, t); t\in\left[0, 1\right]}$}\\
\hline
$\log_2n$ & $\lambda_{\min}$ & $\lambda_{\max}$ & $\kappa$ & $\log_2n$ & $\lambda_{\min}$ & $\lambda_{\max}$ & $\kappa$\\
  \hline
  ${10}$ & 0.6218 & 2.0696 & 3.3285 & ${6}$ & 0.8476 & 1.9916 & 2.3496\\
  ${12}$ & 0.9167 & 3.0511 & 3.3285 & ${7}$ & 1.0298 & 2.4283 & 2.3581\\
  ${14}$ & 1.3514 & 4.4982 & 3.3285 & ${8}$ & 1.2513 & 2.9491 & 2.3569\\
  ${16}$ & 1.9923 & 6.6315 & 3.3285 & ${9}$ & 1.5201 & 3.5804 & 2.3553\\
  \hline
  \hline
${11}$ & 0.0648 & 1.2167 & 18.7767 & ${6}$ & 0.1939 & 1.2180 & 6.2807\\
${12}$ & 0.0648 & 1.3270 & 20.4792 & ${7}$ & 0.1938 & 1.4080 & 7.2655\\
${13}$ & 0.0648 & 1.4373 & 22.1818 & ${8}$ & 0.1938 & 1.5985 & 8.2487\\
${14}$ & 0.0648 & 1.5476 & 23.8843 & ${9}$ & 0.1938 & 1.7893 & 9.2312\\
  \hline
\end{tabular}
\label{tab:sine_3d}
}
\end{center}
\end{table}

In the numerical experiment the range of exponents was limited to $s\in\left[-0.2,
-0.1\right]$ and the upper bound yielded condition numbers independent of the
discretization parameter, cf. Figure \ref{fig:sine_3d}. The observation raises a
question about the suitablity of $s=0$, i.e. considering the multiplier space
$Q_m$ with the $L^2$ norm. It is shown in Remark \ref{rem:s_0} that the choice 
leads to a condition number with logarithmic growth.

% Remark on s=0
\begin{remark}\label{rem:s_0} 
  We consider \eqref{eq:sine_3d_system} with $\Gamma_1$. Since $\la{H}_{0, 0}$ is (due
  to the employed discretization) an identity, the values $S_j$ in \eqref{eq:schur_3d} 
  are the eigenvalues of the preconditioned Schur complement, where $\la{H}_{0, 0}$ is
  the preconditioner. We have $S_j\geq S_n$ and observe that the lower bound sums 
  $\mathcal{O}(n^2)$ terms that are at most $n^{-2}$ in magnitude. Thus $S_n$ is 
  bounded from below by a constant. On the other hand the upper bound 
  $S_j \leq S_1$ grows as $\log{n}$. 

% Note that for the 2$d$-1$d$ trace and $s=0$ we have, cf. Example \ref{ex:sine_2d},
%   \[
% \frac{2}{\pi}\sum_{l\mbox{{ odd}}}^n\frac{1}{n^2+l^2}
%   \leq
%  S_j = \frac{2}{\pi}\sum_{l\mbox{{ odd}}}^n\frac{1}{j^2+l^2}
%   \leq
% \frac{2}{\pi}\sum_{l\mbox{{ odd}}}^n\frac{1}{1+l^2}
%   \leq
%   C,
%   \]
% while the lower bound as a sum of $\mathcal{O}(n)$ terms with $n^{-2}$ magnitude 
% decays as $n^{-1}$. Thus $s=0$ leads to a linearly growing condition number.
  
The estimates for $\Omega\subset\reals^3$ are confirmed by numerical experiments  
summarized in Table \ref{tab:sine_3d}. In particular, the constant lower bound 
and the upper bound growing proportianaly to $\log{n}$, are visible for both configurations.
\end{remark}

Experiments with the spectral discretization suggest that there exists a range
of negative exponents $s$, independent of $\Gamma$, such that the \textit{discrete} trace operator $T_h$ 
defined over $V_h$ can be controlled by the $s$-norm \eqref{eq:snorm} in the sense
of \eqref{eq:cond_cond} and conjecture \ref{thm:cond}. However, the space $V_h$ considered thus far consisted
of infinitely smooth functions. We proceed to show that the statement holds if the
discrete spaces are obtained by FEM. In particular, the space $V_h$ shall be
constructed using the $H^1$ conforming continuous linear Lagrange elements.

\subsection{Trace operator with FEM discretizaton}\label{sec:trace_fem}
Let $V_h\subset H^1(\Omega)$. Further, let $\set{\psi_k}_{k=1}^{m}$ and $\set{L_j}_{j=1}^{m}$ 
be, respectively, the basis and degrees of freedom/dual basis nodal with respect
to $\set{\psi_k}_{k=1}^{m}$
of the finite element space $Q_h$ over $\Gamma$. The trace mapping $T_h:V_h\rightarrow Q_h$
shall be defined by interpolation so that $p_h=T_h u_h$ is represented in the
basis by vector $\la{p}\in\reals^m$,
\begin{equation}\label{eq:Th}
  \la{p}_j = \brack{L_j, u_h|_{\Gamma}}.
\end{equation}
Equivalently we have $\la{p}=\la{T}\la{u}$ where $\la{u}\in\reals^{n}$ and the
matrix representing the trace operator has entries
\begin{equation}\label{eq:trace_entries}
  \la{T}_{i, j} = \brack{L_i, \phi_j|_{\Gamma}},
\end{equation}
where $\set{\phi_j}_{j=1}^{n}$ are the basis functions of $V_h$.

\begin{lemma}[Discrete trace operator by projection]\label{lm:trace} 
Let $u_h\in V_h$ be given and $\tilde{p}_h\in Q_h$ be the $L^2$ projection
\[
  (\tilde{p}_h, q)_{\Gamma} = (u_h|_{\Gamma}, q)_{\Gamma}, \quad q\in Q_h.
\]
Further let $p_h\in Q_h$ be defined via \eqref{eq:Th}.
Then $V_h|_{\Gamma}\subseteq Q_h$ is necessary and sufficient for $p_h=\tilde{p}_h$ .
\end{lemma}
\begin{proof} To verify the assertion let $q_k\in Q_h$ be the Riesz
  representation of $L_k$, i.e. $(q_k, v)_{\Gamma}=\brack{L_k, v}$, $v\in Q_h$,
  and $u_h\in V_h$ arbitrary.
  Then by definition $(p_h, q_k)_{\Gamma}=\brack{L_i, u_h|_{\Gamma}}(\psi_i, q_k)_{\Gamma}$
  and
  \[
  \begin{split}
    \brack{L_i, u_h|_{\Gamma}}(\psi_i, q_k)_{\Gamma}=
    (q_i, u_h|_{\Gamma})_{\Gamma}\brack{L_k, \psi_i}=
    (q_k, u_h|_{\Gamma})_{\Gamma}=
    (q_k, \tilde{p}_h)_{\Gamma}
  \end{split}
  \]
by the property of the Riesz basis ${\set{q_k}_{k=1}^{m}}$, nodality of the basis 
$\set{\psi_i}_{i=1}^{m}$ and definition of $\tilde{p}_h$. It follows that 
$(p_h-\tilde{p}_h, q_k)_{\Gamma}=0$. Note that $u_h|_{\Gamma}\in Q_h$ was 
required to apply the Riesz theorem.
\end{proof}
The above result ensures that $\transp{\la{T}}$ has full column rank, and
consequently the matrix $\la{T}\inv{\la{A}\transp{\la{T}}}$ is non-singular.

\begin{definition}[$\Gamma$-matching spaces]
Let $\Gamma$ be a manifold in $\Omega$ and $Q_h$, $V_h$ the finite element spaces
over the respective domains. The spaces are called $\Gamma$-matching if
  (i) $V_h$ and $Q_h$ are constructed from the same elements and (ii) meshes of 
  $\Omega$ and $\Gamma$ are matched.
\end{definition}
\begin{remark}[Equivalence of interpolation and projection  trace]\label{rm:trace_eq}
  The condition from Lemma \ref{lm:trace} is satisfied with $V_h|_{\Gamma}=Q_h$ if 
  $V_h$ and $Q_h$ are $\Gamma$-matching.
\end{remark}
% Projection <- costs
Finally, note that the interpolation trace is in general cheaper 
to construct than the trace due to projection. We shall employ \eqref{eq:Th}
throughout the rest of the paper. Consequently the trace matrix $\la{T}$ in \eqref{eq:sine_3d_system}
is a product of the mass matrix of the space $Q_h$ and \eqref{eq:trace_entries}.

% The above lemma ensures that $\la{T^T}$ has full column rank, which 
% means that the matrix $\la{T A^{-1} T^T}$ is non-singular. However, 
% in order to construct a preconditioner that is robust with respect to the discretization parameter
% we need a stronger result. Motivated by the theory developed in  
% \cite{2d-1d}, we would need that the continuous trace operator $T: H^1(\Omega) \rightarrow H^{-s}(\Gamma)$
% is an isomorphism for some $s$. From a theoretical point of view, to 
% the authors´ knowledge,  this result is not known. We therefore test numerically whether
% it is possible that there exists such a $s$ and if possible the range of $s$ which would be appropriate.   
 
Let now $V_h$, $Q_h$ be a pair of $\Gamma$-matching spaces constructed from
continuous linear Lagrange elements. Further, the discretization of the geometry
shall be such that the mesh of $\Omega$ is \textit{finer} at/near $\Gamma$ than in
the rest of the domain, cf. Table \ref{tab:fem_mesh} in Appendix \ref{sec:appB} 
and Figure \ref{fig:domains}.
This way the dimensionality of $Q_h$ is increased. Finally, we consider the Schur complement\footnote{
The Schur comeplement is computed from its definition, where the components
$\la{T}$, $\la{A}$ are assembled using FEniCS \cite{logg2012automated,fenics} and PETSc
\cite{petsc} libraries. The Laplacian matrix is then inverted by conjugate
gradient method with algebraic multigrid (AMG) preconditioner from Hypre library
\cite{hypre}. Relative tolerance $10^{-15}$ was set as a convergence criterion.
}
of \eqref{eq:trace_problem} preconditioned by different matrices $\la{H}_{s, 0}$. 
Recall that previously global trigonometric polynomial basis functions were used 
with \eqref{eq:trace_problem} and $-0.2<s\leq -0.1$ yielded condition numbers bounded 
in the discretization parameter. Figure \ref{fig:trace_fem} and Table \ref{tab:trace_fem} 
show that the same conclusions hold also if the finite element discratization is employed.

\begin{figure}
  \begin{center}
  \includegraphics[width=0.49\textwidth]{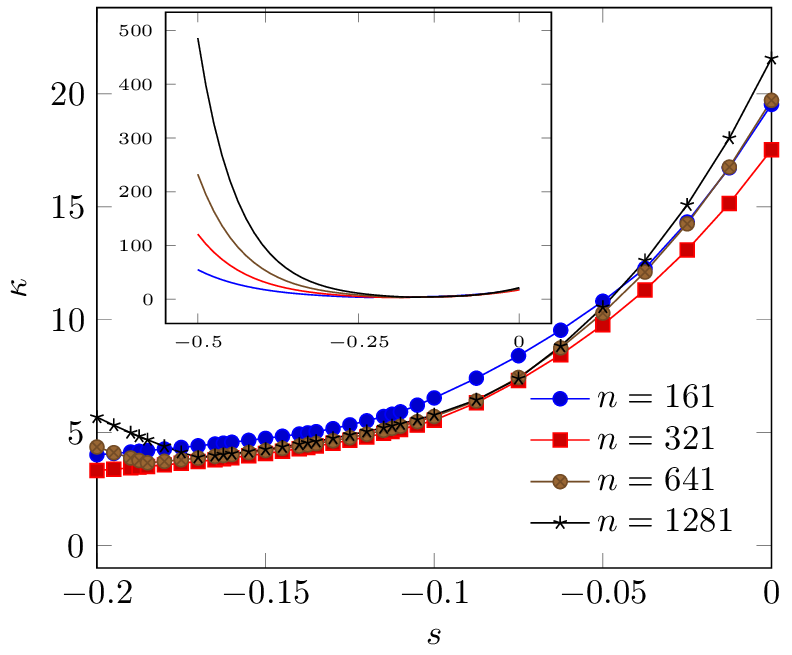}
  \includegraphics[width=0.49\textwidth]{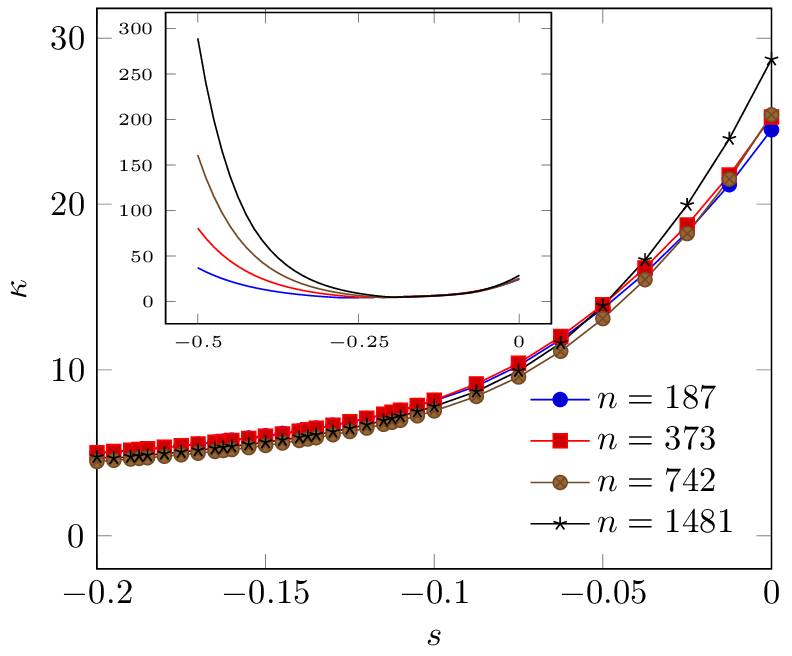}
    \caption{
      Condition numbers \eqref{eq:cond_cond} of \eqref{eq:cond_eigs} with
      finite element discretization, $n=\text{dim} Q_h$, and different
      preconditioners $\la{H}_{s, 0}$. (Left) the curve is $\Gamma_1$. (Right) the
      curve is $\Gamma_2$. The zoomed out plot shows that $s<-0.25$ yields
      unbounded $\kappa$. For both configurations exponents from the interval 
      around $s=-0.1$ yield bounded condition numbers.
    }
    \label{fig:trace_fem}
  \end{center}
\end{figure}
Figure \ref{fig:trace_fem} explores the condition numbers for $s\in\left[-0.5, 0\right]$.
It is evident, cf. the zoom-out plot, that for $s<-0.25$, $\la{H}_{s, 0}$ is not
a good preconditioner for the Schur complement. For both configurations there
are exponents in $(-0.2, -0.1)$ that lead to bounded condition numbers. For
several values of $s$ in this interval, the condition numbers observed on a
sequence of uniformly refined meshes are reported in Table \ref{tab:trace_fem}. 
Therein $s\leq-0.1$ can be observed to lead to bounded $\kappa$. Exponent $s=0$, 
i.e. the $L^2$ norm, leads to a slight growth in $\kappa$ with both $\Gamma_1$ and $\Gamma_2$. 
\begin{table}
  \begin{center}
    \caption{
      Condition numbers \eqref{eq:cond_cond} of \eqref{eq:cond_eigs} for
      selected values of $s$. The finite element
discretization is considered on a sequence of uniformly refined meshes, see
Table \ref{tab:fem_mesh}. For each discretization the mesh is finer near
the curve than in the rest of the domain.
Exponent $s=-0.14$ observed in the spectral discretization, cf. Table
\ref{tab:sine_3d}, yields bounded $\kappa$ also with discrization by FEM. 
Note that similar to the spectral discretization there is a slight growth of $\kappa$ for $s=0$.}
\footnotesize{
\begin{tabular}{c|ccccc||ccccc}
\hline
  \multirow{2}{*}{L\textbackslash $s$}
  & 
\multicolumn{5}{c||}{$\Gamma_1=\set{(t, \tfrac{1}{2}, \tfrac{1}{2}); t\in\left[0, 1\right]}$} 
  & 
\multicolumn{5}{c}{$\Gamma_2=\set{(t, t, t); t\in\left[0, 1\right]}$}\\
\cline{2-11}
  &  -0.16 & -0.14 & -0.12 & -0.1 & 0 & -0.16 & -0.14 & -0.12 & -0.1 & 0\\
\hline
% 2567 & 81 & 4.097 & 4.514 & 5.188 & 6.171 & 18.848\\
1 & 4.568 & 4.932 & 5.517 & 6.531 & 19.530 & 5.760 & 6.316 & 7.064 & 8.129 & 24.484\\
2 & 3.883 & 4.282 & 4.804 & 5.545 & 17.525 & 5.743 & 6.300 & 7.085 & 8.175 & 25.253\\
3 & 4.023 & 4.400 & 4.927 & 5.710 & 19.713 & 5.192 & 5.744 & 6.488 & 7.525 & 25.386\\
4 & 4.062 & 4.477 & 5.045 & 5.781 & 21.561 & 5.381 & 5.926 & 6.698 & 7.798 & 28.731\\
% these are (2567), 12141, 72250, 476443, 36996129
%2003 & 95 & 6.443 & 6.975 & 7.930 & 9.740 & 27.262\\
% these are (2003), 8649, 46486, 307559, 2236026
\hline
\end{tabular}
\label{tab:trace_fem}
}
\end{center}
\end{table}

% Comment on spectral equivalence
We note that in both configurations the behaviour of the eigenvalues is similar to 
the spectral case. In particular, $\lambda_{\max}$ and $\lambda_{\min}$ grow for 
$s\leq-0.1$, whereas for $s=0$ only $\lambda_{\max}$ grows while $\lambda_{\min}$ 
is bounded by a constant, see Table \ref{tab:trace_fem_eigs}. Since the extremal
eigenvalues are in general unbounded $\la{H}_{s, 0}$ is not a
discretization of an operator spectrally equivalent to the Schur complement
and, similar to \S\ref{sec:trace_spectral} the results of FEM disprove
conjecture \ref{thm:equiv}. However, the relation observed in the experiments
\begin{equation}\label{eq:speceq}
  0<\lambda_{\min}(s, h)
  \leq 
  \frac{\transp{\la{x}}\la{T}\inv{\la{A}}\transp{\la{T}}\la{x}}
  {\transp{\la{x}}\la{H}_{s, 0}\la{x}}
  \leq
  \lambda_{\max}(s, h)\quad\forall \la{x}\in\reals^m
\end{equation}
suggests existence of a mesh dependent scale in which spectral equivalence can
be achieved, cf. also results of \S \ref{sec:trace_spectral}. In particular, rescaling
the $s$-norm matrix as $\lambda_{\min}(s, h)\la{H}_{s, 0}$ 
leads to constant bounds, cf. observed constant spectral condition number. We remark 
that $\lambda_{s, \min}$ is bounded away from zero for all $h$ and $s$ observed, in fact the eigenvalue
increases with $h^{-1}$, and in this sense the discrete inf-sup constant never 
approaches zero.

Computational results with the spectral basis and FEM both suggest 
to the construction of the Schur complement preconditioner based on the
mesh dependent $s$-norm $\lambda_{\min}(s, h)\la{H}_{s, 0}$. However, as noted
before, obtaining the scaling factor is
%Based on the mesh-dependent $s$-norm a block-diagonal preconditioner\newline
%$\text{diag}(\la{A}, \lambda_{\min}(h)\la{H}_{s, 0})^{-1}$ could be analysed and
%shown to be optimal using the results of \cite{rusten, silvester1994fast} (see
%also the review paper \cite{pestana}). However, obtaining the scale is
computationally expensive and we shall therefore proceed with \eqref{eq:snorm}
only and not include the scale. In particular, the exponents $s$ identified previously shall
be used to construct preconditioners for several 3$d$-1$d$ constrained problems. We note
that the bounds \eqref{eq:speceq} enter estimates for convergence of
iterative solvers, see, e.g., \cite{silvester1994fast}, and since the bounds here are 
not constant, the proposed preconditioners are theoretically suboptimal. Nevertheless, 
the number of iterations in the studied examples will be bounded. We remark
that the smallest and largest eigenvalues are never far from unity in our examples.

\begin{table}
  \begin{center}
    \caption{
Smallest and largest eigenvalues of the $\la{H}_{s, 0}$ preconditioned Schur 
    complement considered in Table \ref{tab:trace_fem}.
    Similar to spectral discretization both the
extremal eigenvalues grow for $s=-0.14$ while the lower bound is constant and the
upper one grows for $s=0$.
}
\footnotesize{
\begin{tabular}{c|cc||cc}
\hline
\multirow{2}{*}{L}
  & 
\multicolumn{2}{c||}{$\Gamma_1=\set{(t, \tfrac{1}{2}, \tfrac{1}{2}); t\in\left[0, 1\right]}$} 
  & 
\multicolumn{2}{c}{$\Gamma_2=\set{(t, t, t); t\in\left[0, 1\right]}$}\\
\cline{2-5}
  &  $s=-0.14$ & $s=0$ & $s=-0.14$ & $s=0$\\
\hline
% (0.253, 1.143) & (0.053, 1.000) & (0.155, 1.080) & (0.037, 1.000)\\
1 & (0.290, 1.433) & (0.051, 1.000) & (0.207, 1.310) & (0.041, 1.000)\\
2 & (0.420, 1.799) & (0.059, 1.040) & (0.256, 1.610) & (0.041, 1.026)\\
3 & (0.502, 2.208) & (0.059, 1.161) & (0.342, 1.965) & (0.045, 1.145)\\
4 & (0.603, 2.701) & (0.059, 1.276) & (0.401, 2.379) & (0.044, 1.265)\\
\hline
\end{tabular}
\label{tab:trace_fem_eigs}
}
\end{center}
\end{table}

%% file: coupled_rev.tex
\section{Trace coupled problems}\label{sec:coupled}
The previous experiments revealed a range of negative exponents $s$ for which matrices
$\la{H}_s$ behaved similarly to the Schur complement, in terms of stability of the 
condition number, of the related generalized eigenvalue problem. 
To simplify the discussion, we pick $s=-0.14$ and employ the exponent 
to construct preconditioners for two model 3$d$-1$d$ coupled problems.
We note that this choice is somewhat arbitrary and based on \S\ref{sec:trace}
other exponents $s=-0.16$, cf. Table \ref{tab:trace_fem}, could have been
used.

\subsection{Babu\v{s}ka's problem}\label{sec:bab}
Let $V_h$, $Q_h$ be a pair of $\Gamma$-matching spaces constructed by
continuous linear Lagrange elements and consider the problem: 
Find $u\in V_h\subset H^1(\Omega)$, $p\in Q_h\subset H^1(\Gamma)$ such 
that
\begin{equation}\label{eq:bab}
  \begin{aligned}
    (\nabla u, \nabla v)_{\Omega} + (u, v)_{\Omega}+ (p, T v)_{\Gamma} &= (f, v)_{\Omega} &\quad v\in V_h,\\
  (q, T u)_{\Gamma} &= (q, g)_{\Gamma} &\quad q\in Q_h.
  \end{aligned}
\end{equation}
The system \eqref{eq:bab} is a Lagrange multiplier formulation of the 
minimization problem for $v\mapsto \norm{v}^2_{H^1(\Omega)}-2(f, v)_{\Omega}$,
with the constraint $Tv-g=0$ on $\Gamma$. The problem is considered
with homogeneous Neumann boundary conditions. A similar problem with
$\Omega\subset\reals^2$ and $\Gamma\subset\partial\Omega$ was first studied in
\cite{babuska_lm} to introduce Lagrange multipliers as means of prescribing
boundary data.

\begin{figure}
    \begin{minipage}{0.5\textwidth}
      \includegraphics[width=0.45\textwidth]{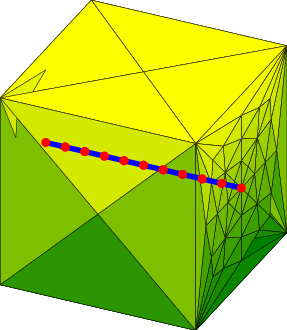}
      \includegraphics[width=0.45\textwidth]{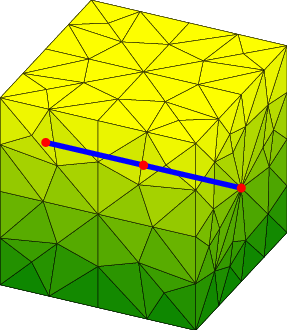}\\
      \includegraphics[width=0.45\textwidth]{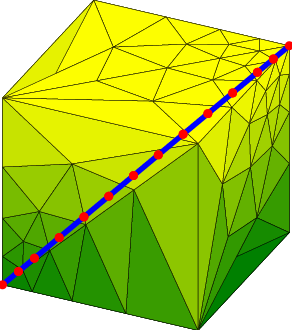}
      \includegraphics[width=0.45\textwidth]{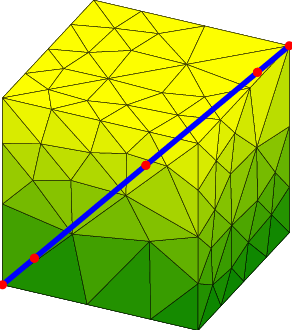}\\
    \end{minipage}%
    \begin{minipage}{0.5\textwidth}
      \includegraphics[width=\textwidth]{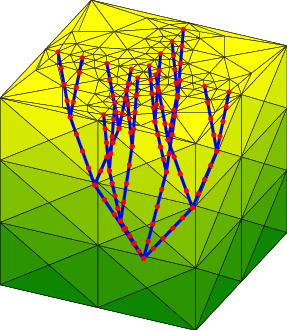}
    \end{minipage}
\caption{Domains used in experiments with matching discretization.
  The one dimensional curve $\Gamma$ is drawn in blue with 
  element boundaries signified by red dots. (Left) The curve is, respectively, a horizontal 
  or diagonal segment. The triangulation of $\Omega$ is either refined or coarsened
  at $\Gamma$. (Right) The curve contains branches and bifurcations, thus
  capturing some of the features of complex vascular systems.
}
\label{fig:domains}
\end{figure}
Similar to the Schur complement study in Section \ref{sec:trace_fem}, the problem shall be considered with
two different curves $\Gamma$. Moreover, for each configuration we consider
three different sequences of uniformly refined meshes, to investigate 
numerically whether the construction of the preconditioner relies on a quasi-uniform 
mesh, or if shape-regular elements are sufficient. 
In a \textit{uniform} discretization the characteristic 
mesh size of $\Omega$ and $\Gamma$ are identical and the tessellation of $\Omega$ is 
structured. In \textit{finer} and \textit{coarser} discretizations the mesh is 
unstructured and is either finer or coarser near $\Gamma$ than in the rest of
the domain. The example meshes are pictured in Figure \ref{fig:domains}. Information 
about the parameters of the discretizations and sizes of
the corresponding finite element spaces are then summarized in Table \ref{tab:fem_mesh}.

Since \eqref{eq:bab} is considered with Neumann boundary conditions, the block
diagonal preconditioner for the system shall have the multiplier block based on
$\la{H}_{s}$ (not $\la{H}_{s, 0}$). We propose the following preconditioned linear system
\begin{equation}\label{eq:bab_precond}
  \inv{\begin{bmatrix}
  \la{A}+\la{M} & \\
              & \la{H}_{-0.14}
  \end{bmatrix}}
\begin{bmatrix}
\la{A}+\la{M}                     & \transp{(\la{M}_{\Gamma}\la{T})}\\
(\la{M}_{\Gamma}\la{T})               &                              \\
\end{bmatrix}
\begin{bmatrix}
  \la{u}\\
  \la{p}\\
\end{bmatrix}
=
  \inv{\begin{bmatrix}
  \la{A}+\la{M} & \\
              & \la{H}_{-0.14}
  \end{bmatrix}}
\begin{bmatrix}
  \la{f}\\
  \la{g}\\
\end{bmatrix},
\end{equation}
%
% Comment on constants?
where $\la{M}$ and $\la{M}_{\Gamma}$ are, respectively, the mass matrices of $V_h$ and
$Q_h$. We remark that the proposed preconditioner is not theoretically optimal
because of the estimate \eqref{eq:speceq}.

In our implementation the leading block of the preconditioner is realized by a
single $V$  cycle of algebraic multigrid from the Hypre\footnote{We have used default values of all the parameters.}
library \cite{hypre}. The system is then solved iteratively with the minimal
residual method (MINRES) implemented in cbc.block \cite{block} and requiring a
preconditioned residual norm smaller than $10^{-12}$ for convergence. 
The initial vectors were random.

The recorded iterations counts are reported in Table \ref{tab:trace}. It can be
seen that the proposed preconditioner results in a bounded number of iterations for all
the considered geometrical configurations and their discretizations. In the
table we also report iteration counts for the preconditioner that employs
$\la{H}_0=\la{M}_{\Gamma}$ for the multiplier block. Recall that with $s=0$ and
spectral discretization, the spectral condition number of the preconditioned Schur 
complement showed a logarithmic growth, cf. Table \ref{tab:sine_3d}. Using FEM, 
the growth was less evident (see Table \ref{tab:trace_fem}), however, the condition 
number was significantly larger than for $s=-0.14$. The iteration counts 
agreee with this observation; the $L^2$ norm leads to at least 20 more iterations.
We remark that the norms in which the convergence criterion is measured differ
between the two cases.
\begin{table}
  \begin{center}
\caption{Iteration counts for preconditioned Babu\v{s}ka's problem \eqref{eq:bab} 
with preconditioners based on \eqref{eq:smat} and $s=-0.14$ or $s=0$ (discrete $L^2$ norm). 
Two geometric configurations and their different discretizations ($L$ denotes the refinement level) 
are considered cf. Figure \ref{fig:domains} and Table \ref{tab:fem_mesh}.
Both preconditioners yield bounded number of iterations. The $L^2$ norm leads to a 
less efficient preconditioner.
}
\footnotesize{
\begin{tabular}{c|ccc||ccc}
\hline
  \multirow{2}{*}{L} 
  &
\multicolumn{3}{c||}{$\Gamma_1=\set{(t, \tfrac{1}{2}, \tfrac{1}{2}); t\in\left[0, 1\right]}$} 
  & 
\multicolumn{3}{c}{$\Gamma_2=\set{(t, t, t); t\in\left[0, 1\right]}$}\\
\cline{2-7}
& uniform & finer & coarser & uniform & finer & coarser \\
\hline
%1 & (29, 44) & (55, 81) & (28, 29) & (29, 44) & (72, 102) & (44, 46)\\
2 & (28, 59) & (53, 81) & (44, 46) & (29, 57) & (73, 107) & (62, 71)\\
3 & (27, 68) & (52, 82) & (49, 58) & (27, 59) & (69, 103) & (64, 81)\\
4 & (25, 70) & (52, 83) & (47, 62) & (25, 61) & (69, 105) & (67, 88)\\
5 & (23, 70) & (53, 83) & (51, 71) & (25, 62) & (70, 105) & (67, 91)\\
\hline
\end{tabular}
\label{tab:trace}
}
\end{center}
\end{table}

\subsection{Model multiphysics problem}\label{sec:multi}
Building upon the Babu\v{s}ka problem we next consider a model multiphysics
problem \eqref{eq:pde}. A similar problem with $\Omega\subset\reals^2$ and 
$\Gamma$ a manifold of codimension one was previously studied by the authors in \cite{2d-1d}. 
Therein it was found that the problem is well posed with the Lagrange multiplier in the
intersection space $H^{\nhalf}(\Gamma)\cap H^{-1}(\Gamma)$. The structure of the
space was mirrored by the preconditioner, which used $\inv{(\la{H}_{-0.5}+\la{H}_{-1})}$
in the corresponding block. 

We note that the exponent $-\tfrac{1}{2}$ was dictated by the properties of the 
continuous trace operator. In the 3$d$-1$d$ case, which is of interest here, we 
shall instead base the exponent/preconditioner on the previous numerical experiments. 
More specifically, the linear system obtained by considering
\eqref{eq:weak_form_h} on finite dimensional finite element subspaces
\begin{equation}\label{eq:coupled_system}
  \begin{bmatrix}
    \la{A}_{\Omega}+\la{M}_{\Omega} & \phantom{0} &\transp{(\la{M}_{\Gamma}\la{T})}\\
    \phantom{0}  & \la{A}_{\Gamma} + \la{M}_{\Gamma} &-\la{M}_{\Gamma}\\
    (\la{M}_{\Gamma}\la{T})      & -\la{M}_{\Gamma}          & \phantom{0}
  \end{bmatrix}
  \begin{bmatrix}\la{u}\\
                 \la{w}\\
                 \la{p}\end{bmatrix}
  =
  \begin{bmatrix}\la{f}\\
                 \la{g}\\
                 \la{h}\end{bmatrix}
\end{equation}
shall be considered with the preconditioner
\begin{equation}\label{eq:coupled_precond}
  \inv{\begin{bmatrix}
    \la{A}_{\Omega}+\la{M}_{\Omega} & \phantom{0} & \phantom{0}\\
    \phantom{0}   & \la{A}_{\Gamma}+\la{M}_{\Gamma} & \phantom{0}\\
    \phantom{0}   &      & \la{H}_{-0.14} + \la{H}_{-1}\\
  \end{bmatrix}}.
\end{equation}
Note that in \eqref{eq:coupled_precond} the structure of the trailing block mimics
the related 2$d$-1$d$ problem. We remark that in the implementation, the remaining
two blocks are realized by AMG. Moreover the discrete spaces are such that
$W_h=Q_h$ and $V_h$, $Q_h$ are $\Gamma$-matching. As in the previous example,
continuous linear Lagrange elements are used.
% Comment on results
To demonstrate the performance of the preconditioner, \eqref{eq:weak_form_h} is
considered on the same geometrical configurations and their discretizations as
\eqref{eq:bab}. The preconditioned system is then solved by MINRES, starting from 
a random initial vector and terminating if the preconditioned residue is less
than $10^{-12}$ in magnitude. As can be seen in Table \ref{tab:coupled}, the 
preconditioner yields bounded iteration counts. Interestingly, the convergence 
is faster on the \textit{finer} discretization than on the \textit{coarser} one. 
We note that the systems on the latter discretization are in general of smaller
size and have more than a factor 10 fewer degrees of freedom in $Q_h$.
However, $\dim Q_h \ll \dim V_h$ is a desirable feature of the model order
reduction which was applied to obtain the problem on $\Gamma$.  
\begin{table}
  \begin{center}
    \caption{Iteration counts for the model problem \eqref{eq:coupled_system}
with preconditioner \eqref{eq:coupled_precond}. Spatial configurations and
    disretizations from Table \ref{tab:trace} are considered. In all the cases 
    the number of iterations is bounded.
}
\footnotesize{
\begin{tabular}{c|ccc||ccc}
\hline
\multirow{2}{*}{L} 
  &
\multicolumn{3}{c||}{$\Gamma_1=\set{(t, \tfrac{1}{2}, \tfrac{1}{2}); t\in\left[0, 1\right]}$} 
  & 
\multicolumn{3}{c}{$\Gamma_2=\set{(t, t, t); t\in\left[0, 1\right]}$}\\
\cline{2-7}
& uniform & finer & coarser & uniform & finer & coarser \\
\hline
%1 & 42 & 48 & 27 & 41 & 63 & 44\\
2 & 51 & 45 & 42 & 44 & 62 & 62\\
3 & 49 & 45 & 48 & 43 & 59 & 62\\
4 & 47 & 43 & 47 & 43 & 59 & 64\\
5 & 46 & 43 & 49 & 42 & 59 & 66\\
\hline
\end{tabular}
\label{tab:coupled}
}
\end{center}
\end{table}

In the examples presented thus far, $\Gamma$ was always a straight segment. To
show that the preconditioner \eqref{eq:coupled_precond} (or the general idea of
$\la{H}_s$ based preconditioners for 3$d$-1$d$ problems) is not limited to such 
simple curves, we shall in the final example consider \eqref{eq:weak_form_h} with 
$\Gamma$ having a more complicated stucture. The considered domain, pictured in
the right pane of Figure \ref{fig:domains}, is inspired by biomechnical applications 
and is intended to mimic some of the features of the vasculature. In particular, the domain
consists of numerous branches and contains multiple bifurcations.

Repeating the setup of the previous experiment, Table \ref{tab:coupled_tree}
reports the iteration counts for the \eqref{eq:coupled_precond} preconditioned linear
system \eqref{eq:coupled_system}, obtained by considering \eqref{eq:weak_form_h} 
on the complex $\Gamma$. The number of iterations is clearly bounded.
% In fact, the number decreases with refinement.

The good performance of the proposed preconditioner in all the considered
examples brings in the question of practicability of its construction. Here, the 
question shall be addressed by considering the setup costs of the preconditioner 
for the domain with complex $\Gamma$. The choice is motivated by the fact that 
(i) the domain is potentially relevant for practical applications and (ii) the 
large (relative to $\dim V_h$) number of degrees of freedom of $Q_h$ puts the
emphasis on the construction of \eqref{eq:smat}. We note that the costs are expected to be 
determined by the multigrid setup and the solution time of the generalized 
eigenvalue problem \eqref{eq:smat}. As in \cite{2d-1d} the eigenvalue problem is
solved by the DSYGVD routine from LAPACK \cite{lapack}. 

The timings obtained on a Linux machine with a single Intel Xeon E5-2680 CPU 
with 2.5GHz and 32GB of RAM are reported in Table \ref{tab:coupled_tree}. The
observed costs of the eigenvalue solve are 3-4 times smaller than that of the multigrid 
setup, and thus the spectral construction does not present a bottleneck. Morover, 
both AMG and GEVP are expected to scale roughly as $\dim{Q_h}^3$. However, due
to the cubic scaling, the system/preconditioner is unlikely to be assembled/setup in
serial. For such a case, a scalable parallel implementation, for the construction of \eqref{eq:smat}, 
remains an issue, and approaches that provide the approximate action of $\la{H}_s$ 
matrices may offer better performance. Examples of such approaches are
the \cite{arioli_siam, arioli_ima} and \cite{harizanov2016optimal} where
polynomial and rational function approximations are constructed, fast Fourier transforms \cite{fft}
or methods \cite{contour, bonito2015numerical} based on integral definitions of fractional Laplacian
\cite{kwasnicki2017ten}.

\begin{table}
  \begin{center}
    \caption{
Iteration counts and setup costs (in seconds) for system \eqref{eq:coupled_system} 
and preconditioner \eqref{eq:coupled_precond}. Both operators are assembled for 
the complex $\Gamma$ pictured in Figure \ref{fig:domains}. The number of iterations 
is bounded in the discretization parameter. In the considered example, the eigenvalue
(GEVP) based construction \eqref{eq:smat} does not present a bottleneck as it is 3-4
times cheaper than setting up the algebraic multigrid (AMG).
}
\footnotesize{
\begin{tabular}{ccc|cc}
\hline
  $\dim{V_h}$ & $\dim{Q_h}$ & \# & $\text{AMG}\left[s\right]$ & $\text{GEVP}\left[s\right]$\\
\hline
%1K & 117 & 106 & 0.0 & 0.0\\
%1K & 213 & 88 & 0.0 & 0.0\\
%4K   & 413  & 91 & 0.0   & 0.0\\
18K  & 817  & 86 & 0.2   & 0.1\\
100K & 1605 & 81 & 1.9   & 0.6\\
634K & 3193 & 76 & 15.0  & 4.2\\
4.8M & 6381 & 68 & 141.6 & 36.4\\
\hline
\end{tabular}
\label{tab:coupled_tree}
}
\end{center}
\end{table}

%% file: unmatch_rev.tex
\section{Nonmatching discrete trace}\label{sec:unmatch}
% Intro
The numerical examples presented thus far have always employed $\Gamma$-matching
finite element spaces. We note that in \cite{2d-1d} this construction is shown
to imply that the discrete inf-sup condition holds for problems \eqref{eq:bab} and 
\eqref{eq:weak_form_h} considered with $\Omega\subset\reals^2$ and $\Gamma$ a 
one dimensional curve. However, the assumption of matched discretizations of
$\Omega$ and $\Gamma$ can be too limiting, e.g, if fine resolution is requested 
on the curve. In this section we present numerical examples using the Babu\v{s}ka 
problem \eqref{eq:bab}, which demonstrate that the matching discretization assumption
is not necessary and to the extent given by the new inf-sup condition the discretizations can be 
independent. Using such stable discretizations and preconditioners based on characterization 
of the trace the observed number of Krylov iterations will remain bounded.
%We note that
%from the point of view of Lemma \ref{lm:trace} the spaces shall be such that $V_h|_{\Gamma}\supset Q_h$.

\subsection{Codimension 1}\label{sec:uncut_cod1}
Consider \eqref{eq:bab} with $\Omega\subset\reals^2$. 
For $\Gamma\subset\partial\Omega$, the finite element discretization of the
problem requires that the spaces $V_h$, $Q_H$ (we use different subscripts to
indicate the difference in underlying triangulations) are such that $h\leq cH$ for
some $c<1$. Here $h$ is understood as a mesh size of $V_h$ on $\Gamma$. The
inequality ensures that the discrete inf-sup condition is satisfied
\cite{steinbach2007numerical, Dahmen2001}.
%We note that \cite{pitkaranta_lm} 
%shows that the inequality is not necessary.

Let now $\Gamma$ be a curve, contained in $\Omega$, where the domains are discretized 
such that the condition from the previous paragraph is met. Further, the space
$V_h$ shall be discretized by continuous linear Lagrange elements, while, for the
construction of $Q_H$, either the same elements or piecewise 
constant Lagrange elements are employed. We note that with the latter choice the
eigenvalue problem for the discrete $s$-norm simplifies, since the mass matrix is
diagonal in this case.

\begin{figure}
    \begin{minipage}{0.22\textwidth}
      \includegraphics[width=\textwidth]{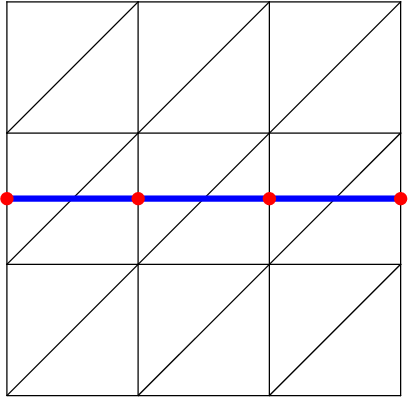}\\
      \includegraphics[width=\textwidth]{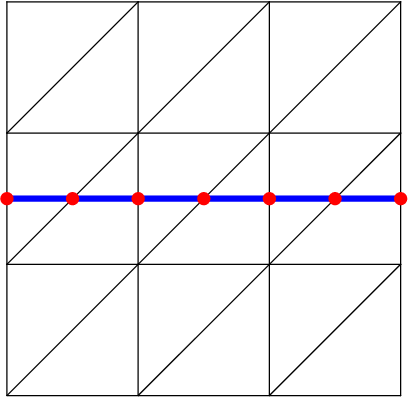}
    \end{minipage}
    \hspace{5pt}
    \begin{minipage}{0.73\textwidth}
      \includegraphics[width=0.48\textwidth]{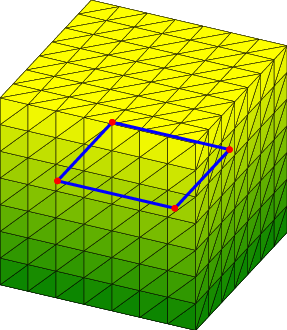}
      \includegraphics[width=0.48\textwidth]{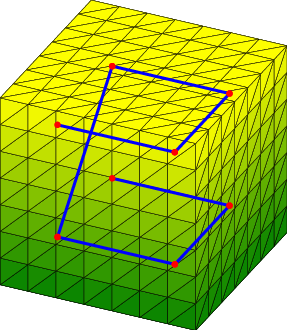}
    \end{minipage}
  \caption{Domains used in experiments with nonmatching discretization. (Left) The spaces $V_h$ and $Q_H$ 
  are inf-sup stable for \eqref{eq:bab} if $h\leq cH$, $c<1$. The condition is satisfied/violated 
  in the top/bottom configurations.
  (Right) The 3$d$-1$d$ experiments use two curves $\Gamma$. The mesh of $\Omega$ is 
  obtained by first subdividing the domain into odd number of cubes in each direction. 
  Thus degrees of freedom of $V_h$, $Q_H$ are not associated with identical spatial points.
  Moreover $h\ll H$ is ensured in the refinement.
  }
\label{fig:domains_cut}
\end{figure}
Table \ref{tab:uncut_2d_circle} reports the number of MINRES iterations on the
system \eqref{eq:bab}, using $\text{diag}(\text{AMG}(\la{A}+\la{M}), \inv{\la{H}_{-0.5}})$
as the preconditioner. The iterations are started from a random vector using
$10^{-12}$ as the stopping tolerance for the magnitude of the preconditioned
residual. With both considered finite element discretizations of the multiplier
space the number of iterations is bounded indicating (i) that the inf-sup condition 
is satisfied and (ii) the optimality of the preconditioner. We note that for $h>H$,
the inf-sup condition is violated and in turn the the iterations are unbounded (not reported here).
An example of a pair of inf-sup stable and unstable discretizations 
is shown in Figure \ref{fig:domains_cut}.

\begin{table}[ht!]
  \begin{center}
    \caption{
Iteration counts and error convergence for \eqref{eq:bab} and $\Omega$ a unit
square and $\Gamma$ a circle. The spaces $V_h$ and $Q_H$ are formed either by continuous 
linear Lagrange elements or $Q_H$ 
uses discontinuous piecewise constant Lagrange elements. Note that $\Gamma$ is closed 
and thus $Q_H$ has the same dimension with either of the elements. The inequality 
$h\leq cH$, $c<1$ is respected ensuring that the inf-sup condition \cite{steinbach2007numerical, Dahmen2001} is satisfied. 
Consequently the iteration count is bounded. Both pairs yield optimal, order 1, 
convergence in $H^1(\Omega)$ norm of the error $u-u_h$. We note that the exact 
solution is smooth. The error of the Lagrange multiplier measured in the 
$s=-\tfrac{1}{2}$ norm \eqref{eq:snorm} (computed on $Q_H$) norm decays with order 1.5.
}
\footnotesize{
\begin{tabular}{cc|ccc||ccc}
\hline
  \multirow{2}{*}{$\dim V_h$} 
  &
  \multirow{2}{*}{$\dim Q_H$} 
  &
  \multicolumn{3}{c||}{$Q_H$ continuous} 
  & 
  \multicolumn{3}{c}{$Q_H$ discontinuous}\\
\cline{3-8}
  & & \# & $\norm{u-u_h}_V$ & $\norm{p-p_h}_{Q}$ & 
      \# & $\norm{u-u_h}_V$ & $\norm{p-p_h}_{Q}$ \\
\hline
%371 & 17 & 43 & 7.97E-01 & 3.66E-01 & 43 & 7.75E-01 & 1.14E-01\\
%1K & 34 & 49 & 3.91E-01 & 6.70E-02 & 47 & 3.88E-01 & 2.61E-02\\
%6K & 68 & 50 & 1.94E-01 & 1.83E-02 & 45 & 1.94E-01 & 1.08E-02\\
22K & 136 & 52 & 9.54E-02 & 5.28E-03 & 47 & 9.54E-02 & 3.68E-03\\
87K & 272 & 52 & 4.78E-02 & 1.71E-03 & 48 & 4.78E-02 & 1.15E-03\\
348K & 544 & 51 & 2.39E-02 & 5.77E-04 & 49 & 2.39E-02 & 4.18E-04\\
1.4M & 1088 & 51 & 1.19E-02 & 1.87E-04 & 50 & 1.19E-02 & 1.49E-04\\
\hline
\end{tabular}
\label{tab:uncut_2d_circle}
}
\end{center}
\end{table}

\subsection{Codimension 2} Due to the difficulties with the trace
operator for $\Gamma$ a manifold of codimension two, cf. \S \ref{sec:prelims},
the functional setting of \eqref{eq:bab} is not clear and therefore corresponding
discrete inf-sup conditions for the problem is not available. However, we shall assume
that the inequality $h\leq cH$, $c<1$, which was cruacial for the 2$d$-1$d$ problems,
plays a role also in the 3$d$-1$d$ case and discretize the domains  accordingly.

The problem \eqref{eq:bab} is considered with two carefully constructed curves 
$\Gamma$, see Figure \ref{fig:domains_cut}, and $\Omega$ a unit cube discretized 
such that the inequality is ensured. As before, the spaces $Q_H$ are constructed 
from continuous piecewise linear or discontinuous piecewise constant Lagrange elements. 
We note that $\dim Q_h\ll \dim V_h$. Further, the MINRES iterations use the 
same initial and convergence conditions as in \S\ref{sec:uncut_cod1}, while
$\text{diag}(\text{AMG}(\la{A}+\la{M}), \inv{\la{H}_{-0.14}})$ is used as the
preconditioner. In Table \ref{tab:uncut_3d} we observe that the discretization
and the preconditioner lead to bounded iteration counts.
We note that if the discretization of $\Gamma$ violates the inequality
$h < c H$, the number of iterations cannot be bounded anymore.

\begin{table}[ht!]
  \begin{center}
    \caption{
Iteration counts for \eqref{eq:bab_precond} posed on $\Omega\subset\reals^3$ 
and the two curves pictured in Figure \ref{fig:domains_cut}. For each domain,
$Q_H$ from continuous linear (first column) or discontinuous constant (second column) 
Lagrange elements is considered. The domains are discretized such that
$h\leq cH$, $c<1$. In all the cases, the number of iterations is bounded.
}
\footnotesize{
\begin{tabular}{c|cc|cc||cc|cc}
\hline
  \multirow{2}{*}{$\dim V_h$} 
  &
\multicolumn{4}{c||}{Square} 
  & 
\multicolumn{4}{c}{Spiral}\\
\cline{2-9}
  & $\dim Q_H$ & \# & $\dim Q_H$ & \# & $\dim Q_H$ & \# & $\dim Q_H$ & \#\\
\hline
%1K & 4 & 14 & 4 & 14 & 8 & 23 & 7 & 22\\
%4K & 8 & 24 & 8 & 20 & 15 & 38 & 14 & 27\\
33K & 16 & 36 & 16 & 24 & 29 & 48 & 28 & 36\\
262K & 32 & 38 & 32 & 24 & 57 & 48 & 56 & 35\\
2.1M & 64 & 36 & 64 & 23 & 113 & 46 & 112 & 35\\
6.0M & 128 & 38 & 128 & 24 & 225 & 48 & 224 & 36\\
\hline
\end{tabular}
\label{tab:uncut_3d}
}
\end{center}
\end{table}

%% file: conclusions_rev.tex
\section{Conclusions}\label{sec:conclusions}
We have discussed preconditioning of a model multiphysics problem \eqref{eq:pde}, 
where two elliptic subproblems were coupled by a trace constraint, bridging the
dimensionality gap of size two. In order to facilitate the re-use of standard
multilelevel preconditioners for the 3$d$ domain we considered the trace as a
mapping from $H^1(\Omega)$ to $H^s(\Gamma)$ for some $s<0$ and consequently
conjectured that the Schur complement of \eqref{eq:pde} is related to the
fractional Laplacian $(-\Delta)^s$. Using a simpler problem \eqref{eq:trace_problem}
the spectral equivalence was investigated by a series of numerical experiments
revealing for $s\in (-0.2, -0.1)$ existence of a mesh-dependent scale $\tau(s, h)$
such that $\tau(s, h)(-\Delta_h)^s$ is a robust preconditioner for the Schur complement.
As the scale is, in general, impractical to compute only the fractional Laplacian
was further used in preconditioning the coupled problem \eqref{eq:pde}.
Robustness of the proposed preconditioner was demonstrated by numerical
experiments with curves of different complexity and various shape-regular
meshes using, at first, the assumption $V_h|_{\Gamma}=Q_h$ and finally with
spaces $V_h$, $Q_H$ satisfying the compatibility condition $h\leq cH$, $c<1$
inspired by 2$d$-1$d$ problems \cite{steinbach2007numerical, Dahmen2001}.

%% file: appendix_rev.tex
\section{Geometrical configurations and their discretization}\label{sec:appB}
Numerical experiments with the Schur complement in \S\ref{sec:trace_fem} and the
coupled problem in \S\ref{sec:coupled} are considered on sequences of uniformly
refined meshes, discretizing the geometrical configurations shown in Figure
\ref{fig:domains}. The Schur complement experiment is considered with straight
segments $\Gamma_1=\set{(t, \tfrac{1}{2}, \tfrac{1}{2}); t\in\left[0, 1\right]}$ 
or $\Gamma_2=\set{(t, t, t); t\in\left[0, 1\right]}$. For each case the domains
are discretized in three ways: (\textit{uniform}) the meshes for $\Omega$, $\Gamma$ have 
the same characteristic size, (\textit{finer}) the mesh of $\Omega$ is finer at $\Gamma$ than 
in the rest of the domain, (\textit{coarser}) the mesh of $\Omega$ is coarser at $\Gamma$ than 
in the rest of the domain. Parameters of the meshes for each refinement level
are summarized in Table \ref{tab:fem_mesh}.

\begin{table}[ht!]
  \begin{center}
    \caption{
Sizes of FEM spaces and mesh parameters for different levels of
refinements ($L$). The length of the largest cell in the mesh of $\Gamma_i$ is
    denoted by $H$. For readability the reported value is $H\times 10^{3}$.
Lengths of smallest/largest edges of cells of the mesh for 
$\Omega\setminus\Gamma_i$ are respectively $h_{\min}$ and $h_{\max}$. (Top) In 
    \textit{uniform}
discretization the characteristic mesh size of $\Omega$ and $\Gamma_i$ triangulations 
    are identical. (Middle) \textit{Finer} discretization uses finer mesh near $\Gamma_i$.
    (Bottom) In the \textit{coarser} cases the mesh of $\Omega$ is coarser near 
    the curve.
}
\footnotesize{
\begin{tabular}{c|ccccc||ccccc}
\hline
\multirow{2}{*}{L} 
  &
\multicolumn{5}{c||}{$\Gamma_1=\set{(t, \tfrac{1}{2}, \tfrac{1}{2}); t\in\left[0, 1\right]}$} 
  & 
\multicolumn{5}{c}{$\Gamma_2=\set{(t, t, t); t\in\left[0, 1\right]}$}\\
\cline{2-11} & 
$\dim{V_h}$ & $\dim{Q_H}$ & $\tfrac{h_{\min}}{H}$ & $\tfrac{h_{\max}}{H}$ & $H$ & 
$\dim{V_h}$ & $\dim{Q_H}$ & $\tfrac{h_{\min}}{H}$ & $\tfrac{h_{\max}}{H}$ & $H$\\
\hline
%1K & 9 & 1.7 & 1.7 & 125.0 & 1K & 9 & 1.0 & 1.0 & 216.5\\
1& 5K & 17 & 1.7 & 1.7 & 62.5 & 5K & 17 & 1.0 & 1.0 & 108.3\\
2& 36K & 33 & 1.7 & 1.7 & 31.2 & 36K & 33 & 1.0 & 1.0 & 54.1\\
3& 275K & 65 & 1.7 & 1.7 & 15.6 & 275K & 65 & 1.0 & 1.0 & 27.1\\
4& 2.1M & 129 & 1.7 & 1.7 & 7.8 & 2.1M & 129 & 1.0 & 1.0 & 13.5\\
5& 6.1M & 183 & 1.7 & 1.7 & 5.5 & 6.1M & 183 & 1.0 & 1.0 & 9.5\\
\hline
\hline
%3K & 81 & 1.3 & 24.0 & 12.5 & 2K & 95 & 0.9 & 19.5 & 18.7\\
1& 12K & 161 & 1.1 & 32.9 & 6.2 & 9K & 187 & 1.0 & 22.5 & 9.4\\
2& 72K & 321 & 1.0 & 35.3 & 3.1 & 46K & 373 & 0.9 & 24.7 & 4.7\\
3& 476K & 641 & 0.9 & 39.0 & 1.6 & 308K & 742 & 0.8 & 27.3 & 2.3\\
4& 3.7M & 1281 & 0.8 & 40.6 & 0.8 & 2.2M & 1481 & 0.8 & 27.0 & 1.2\\
5& 6.8M & 1601 & 0.7 & 40.8 & 0.6 & 7.4M & 2220 & 0.8 & 27.0 & 0.8\\
\hline
\hline
1& 11K & 9 & 0.2 & 1.7 & 125.0 & 5K & 16 & 0.2 & 1.7 & 122.5\\
2& 59K & 17 & 0.2 & 1.9 & 62.5 & 30K & 31 & 0.2 & 2.0 & 61.2\\
3& 375K & 33 & 0.2 & 2.1 & 31.2 & 194K & 59 & 0.2 & 2.2 & 30.6\\
4& 2.7M & 65 & 0.2 & 2.1 & 15.6 & 1.4M & 114 & 0.2 & 2.3 & 15.5\\
5& 8.5M & 97 & 0.2 & 2.5 & 10.4 & 4.4M & 169 & 0.2 & 3.2 & 10.4\\
\hline
\end{tabular}
\label{tab:fem_mesh}
}
\end{center}
\end{table}

% \begin{table}[ht!]
%   \begin{center}
%     \caption{
%       Dimensions $(V_h, Q_h)$ as functions of mesh refinement level $L$. In
%       \textit{uniform} case the characteristic mesh size of $\Omega$ and
%       $\Gamma$ triangulations are identical. The mesh size for $\Omega$ outside
%       $\Gamma$ is larger and smaller respectively in $\textit{finer}$ and
%       $\textit{coarser}$ cases.
% }
% \footnotesize{
% \begin{tabular}{c|ccc||ccc}
% \hline
% \multirow{2}{*}{L} 
%   &
% \multicolumn{3}{c||}{$\Gamma=\set{(t, \tfrac{1}{2}, \tfrac{1}{2}); t\in\left[0, 1\right]}$} 
%   & 
% \multicolumn{3}{c}{$\Gamma=\set{(t, t, t); t\in\left[0, 1\right]}$}\\
% \cline{2-7}
% & uniform & finer & coarser & uniform & finer & coarser \\
% \hline
% 1 & (5K, 17)    & (12K, 161)   & (11K, 9)   & (5K, 17)    & (9K, 187)    & (5K, 16)\\
% 2 & (36K, 33)   & (72K, 321)   & (59K, 17)  & (36K, 33)   & (46K, 373)   & (30K, 31)\\
% 3 & (275K, 65)  & (476K, 641)  & (375K, 33) & (275K, 65)  & (308K, 742)  & (194K, 59)\\
% 4 & (2.1M, 129) & (3.7M, 1281) & (2.7M, 65) & (2.1M, 129) & (2.2M, 1481) & (1.4M, 114)\\
% 5 & (6.1M, 183) & (6.8M, 1601) & (8.5M, 97) & (6.0M, 182) & (7.4M, 2220) & (4.4M, 169)\\
% \hline
% \end{tabular}
% \label{tab:fem_mesh}
% }
% \end{center}
% \end{table}
% %\KAM{Need to say something about hmax/hmin}